\documentclass[a4paper, 10pt, parskip=half]{scrartcl}

\usepackage[utf8]{inputenc}
\usepackage[T1]{fontenc}
\usepackage{lmodern}
\usepackage{csquotes}
\usepackage{amsmath}
\usepackage{amssymb}
\usepackage{amsthm}
\usepackage{amsfonts}
\usepackage{dsfont}
\usepackage{mathtools}
\usepackage{marginnote}
\usepackage{hyperref}
\hypersetup{%
  colorlinks=true,
  linkcolor=black,
  citecolor=black,
  urlcolor=black,
  pdftitle={Global solutions near homogeneous steady states in a multi-dimensional population model with both predator- and prey-taxis},
  pdfauthor={Mario Fuest},
  pdfkeywords={double cross diffusion; large-time behavior; predator--prey; stability},
  bookmarksopen=true,
}
\newcommand{\tops}{\texorpdfstring}

\usepackage[numbers, sort&compress]{natbib}
\RequirePackage{geometry}
\newcommand{\R}{\mathbb{R}}

\newcommand{\N}{\mathbb{N}}

\newcommand{\mc}[1]{\mathcal{#1}}

\newcommand{\ur}[1]{\mathrm{#1}}
\newcommand{\ure}{\ur{e}}

\ifdefined\labelenumi%
  \renewcommand{\labelenumi}{(\roman{enumi})}
  
\fi

\newcommand{\eps}{\varepsilon}

\newcommand{\gt}{>}
\newcommand{\lt}{<}

\DeclareMathOperator{\tr}{tr}

\newcommand{\defs}{\coloneqq}
\newcommand{\sfed}{\eqqcolon}

\newcommand{\ra}{\rightarrow}

\newcommand{\nea}{\nearrow}

\newcommand{\rh}{\rightharpoonup}

\newcommand{\ol}{\overline}

\newcommand{\ds}{\,\mathrm{d}s}

\newcommand{\ddt}{\frac{\mathrm{d}}{\mathrm{d}t}}

\newcommand{\complete}[2]{\ensuremath\ol{#1}^{\|\cdot\|_{#2}}}

\newcommand{\embed}{\hookrightarrow}

\newcommand{\hp}{\hphantom}
\newcommand{\pe}{\mathrel{\hp{=}}}

\newcommand{\tmax}{T_{\max}}

\newcommand{\intom}{\int_\Omega}

\newcommand{\Ombar}{\ol \Omega}

\newcommand{\leb}[1]{\ensuremath{L^{#1}(\Omega)}}
\newcommand{\sob}[2]{\ensuremath{W^{#1, #2}(\Omega)}}
\newcommand{\sobn}[2]{\ensuremath{W_N^{#1, #2}(\Omega)}}

\newcommand{\con}[1]{\ensuremath{C^{#1}(\Ombar)}}

\newcommand{\ustar}{u_\star}
\newcommand{\vstar}{v_\star}
\newcommand{\cp}{C_{\mathrm P}}

\makeatletter
\renewenvironment{proof}[1][\proofname]{\par
  \pushQED{\qed}%
  \normalfont \topsep0\p@\relax
  \trivlist
  \item[\hskip\labelsep\scshape
  #1\@addpunct{.}]\ignorespaces
}{%
  \popQED\endtrivlist\@endpefalse
}
\makeatother

\newtheorem{base}{Base}[section]
\numberwithin{equation}{section}

\newtheorem{theorem}[base]{Theorem} \newtheorem*{theorem*}{Theorem}
\newtheorem{lemma}[base]{Lemma} \newtheorem*{lemma*}{Lemma}
 \newtheorem*{prop*}{Proposition}
 \newtheorem*{cor*}{Corollary}

\theoremstyle{definition}
\newtheorem{remark}[base]{Remark} \newtheorem*{remark*}{Remark}

\theoremstyle{definition}
 \newtheorem*{definition*}{Definition}
 \newtheorem*{example*}{Example}
 \newtheorem*{cond*}{Condition}

\begin{document}
\setkomafont{title}{\normalfont\Large}
\title{Global solutions near homogeneous steady states in a multi-dimensional population model with both predator- and prey-taxis}
\author{%
Mario Fuest\footnote{mail: \href{mailto:fuestm@math.uni-paderborn.de}{fuestm@math.uni-paderborn.de}, ORCID: \href{https://orcid.org/0000-0002-8471-4451}{0000-0002-8471-4451}}\\
{\small Institut f\"ur Mathematik, Universit\"at Paderborn,}\\
{\small 33098 Paderborn, Germany}
}
\date{}

\maketitle

\KOMAoptions{abstract=true}
\begin{abstract}
\noindent
We study the system
\begin{align*}\label{prob:star} \tag{$\star$}
  \begin{cases}
    u_t = D_1 \Delta u - \chi_1 \nabla \cdot (u \nabla v) + u(\lambda_1 - \mu_1 u + a_1 v) \\
    v_t = D_2 \Delta v + \chi_2 \nabla \cdot (v \nabla u) + v(\lambda_2 - \mu_2 v - a_2 u)
  \end{cases}
\end{align*}
(inter alia) for $D_1, D_2, \chi_1, \chi_2, \lambda_1, \lambda_2, \mu_1, \mu_2, a_1, a_2 > 0$
in smooth, bounded domains $\Omega \subset \mathbb R^n$, $n \in \{1, 2, 3\}$.\\[0.5pt]
Without any further restrictions on these parameters,
we prove that there exists a constant stable steady state $(u_\star, v_\star) \in [0, \infty)^2$,
meaning that there is $\varepsilon > 0$ such that,
if $u_0, v_0 \in W^{2, 2}(\Omega)$ are nonnegative with $\partial_\nu u_0 = \partial_\nu v_0 = 0$ in the sense of traces and
\begin{align*}
  \|u_0 - u_\star\|_{W^{2,2}(\Omega)} + \|v_0 - v_\star\|_{W^{2,2}(\Omega)} < \varepsilon,
\end{align*}
then there exists a global classical solution $(u, v)$ of~\eqref{prob:star} with initial data $u_0, v_0$
converging to $(u_\star, v_\star)$ in~$W^{2, 2}(\Omega)$.
Moreover, the convergence rate is exponential,
except for the case $\lambda_2 \mu_1 = \lambda_1 a_2$, where it is is only algebraical.\\[0.5pt]
To the best of our knowledge,
this constitutes the first global existence result for \eqref{prob:star} in the biologically most relevant two- and three-dimensional settings.
In the proof, we make use of the special structure in \eqref{prob:star} and carefully balance the doubly cross-diffusive interaction therein.
Indeed, we introduce certain functionals and combine them in a way allowing for cancellations of the most worrisome terms.\\[0.5pt]
 \textbf{AMS Classification (2020):} {35B35 (primary); 35K55, 35K57, 92D25 (secondary)} \end{abstract} 
\section{Introduction}
Migration-influenced predator–prey interaction can mathematically be described by the system
\begin{align}\label{prob:general}
  \begin{cases}
    u_t = D_1 \Delta u + \nabla \cdot (\rho_1(u, v) \nabla v) + f(u, v), \\
    v_t = D_2 \Delta v + \nabla \cdot (\rho_2(u, v) \nabla u) + g(u, v).
  \end{cases}
\end{align}
Therein, $u$ and $v$ model the density of predators and prey, respectively.
Apart from growth, death or intra-species competition, the functions $f$ and $g$ model predation:
Encounters are beneficial for the predators and harmful to the prey.
Moreover, the species are not only assumed to move around randomly (terms $D_1 \Delta u$ and $D_2 \Delta v$),
but also to be able to direct their movement toward (attractive taxis, negative $\rho_i$)
or away from (repulsive taxis, positive $\rho_i$) higher concentration of the other species.

The relevance of attractive prey-taxis (`predators move towards their prey', negative $\rho_1$)
has first been biologically verified in~\cite{KareivaOdellSwarmsPredatorsExhibit1987}.
It has been observed that such an effect may actually reduce effective biocontrol,
contradicting intuitive assumptions~\cite{LeeEtAlContinuousTravelingWaves2008}.
Moreover, the presence of (sufficiently strong) prey-taxis may actually lead to a lack of
pattern formation~\cite{LeeEtAlPatternFormationPreytaxis2009}.

Among systems of the form~\eqref{prob:general},
those with only attractive prey- but no predator-taxis ($\rho_1 \lt 0$ and $\rho_2 \equiv 0$),
have been studied most extensively---%
perhaps because they resemble attractive chemotaxis systems from a mathematical point of view,
which in turn have been studied in comparatively great detail; see for instance the survey~%
\cite{BellomoEtAlMathematicalTheoryKeller2015}.

For $\rho_1(u, v) = -\chi u$ and several $f, g$, namely,
the existence of globally bounded classical solutions to \eqref{prob:general}
has been proved in \cite{WuEtAlGlobalExistenceSolutions2016},
provided $\chi \gt 0$ is sufficiently small.
In two space dimensions, the smallness condition on $\chi$ is, again for various choices of $f$ and $g$, not necessary~%
\cite{JinWangGlobalStabilityPreytaxis2017, XiangGlobalDynamicsDiffusive2018},
while in the three-dimensional setting, one may overcome this restriction
by either assuming the prey-taxis to be saturated at larger predator quantities~%
\cite{HeZhengGlobalBoundednessSolutions2015, TaoGlobalExistenceClassical2010}
or by considering weak solutions instead~\cite{WinklerAsymptoticHomogenizationThreedimensional2017}.

Moreover, a repulsive predator-taxis mechanism (`prey moves away from their predators, positive $\rho_2$) 
has, for instance, been detected for crayfish seeking shelter~%
\cite{GarveyEtAlAssessingHowFish1994, HillLodgeReplacementResidentCrayfishes1999, LeeEtAlContinuousTravelingWaves2008}.

While less extensively studied than those with prey-taxis, such systems have been mathematically examined as well:
Now without any smallness assumptions on $\chi$,
globally bounded classical solutions to~\eqref{prob:general} have been constructed 
for $\rho_1 \equiv 0$, $\rho_2(u, v) = \chi v$ and certain $f, g$ 
in~\cite{WuEtAlDynamicsPatternFormation2018}.
The same article also considered pattern formation and shows that a strong taxis mechanism (large $\chi$)
leads to the absence of stable nonconstant steady states.

Combining both these effects $(\rho_1 \lt 0$, $\rho_2 \gt 0$)
leads to the study of so-called pursuit--evasion models
which have been proposed in~\cite{TsyganovEtAlQuasisolitonInteractionPursuitEvasion2003}
(see also~\cite{GoudonUrrutiaAnalysisKineticMacroscopic2016, TyutyunovEtAlMinimalModelPursuitEvasion2007}
for the modelling of related systems featuring different taxis mechanisms).
There, propagating waves differing from those in taxis-free predator--prey systems have been detected numerically.

\paragraph{Main results.}
In the present article, we handle a system including both predator- and prey-taxis
and take the prototypical choices
$\rho_1(u, v) = -\chi_1 u$, $\rho_2(u, v) = \chi_2 v$,
$f(u, v) = u (\lambda_1 - \mu_1 u + a_1 v)$ and $g(u, v) = v (\lambda_2 - \mu_2 v - a_1 u)$ for $u, v \ge 0$ in \eqref{prob:general}.
That is, we consider
\begin{align}\label{prob}
  \begin{cases}
    u_t = D_1 \Delta u - \chi_1 \nabla \cdot (u \nabla v) + u(\lambda_1 - \mu_1 u + a_1 v), & \text{in $\Omega \times (0, \infty)$}, \\
    v_t = D_2 \Delta v + \chi_2 \nabla \cdot (v \nabla u) + v(\lambda_2 - \mu_2 v - a_2 u), & \text{in $\Omega \times (0, \infty)$}, \\
    \partial_\nu u = \partial_\nu v = 0,                                                    & \text{on $\partial \Omega \times (0, \infty)$}, \\
    u(\cdot, 0) = u_0, v(\cdot, 0) = v_0,                                                   & \text{in $\Omega$}
  \end{cases}
\end{align}
in smooth, bounded domains $\Omega$
for $D_1, D_2, \chi_1, \chi_2 \gt 0$ and $\lambda_1, \lambda_2, \mu_1, \mu_2, a_1, a_2 \ge 0$.

From a mathematical point of view, such systems are much more challenging than those containing a taxis term in `only' one equation,
which are in turn already highly non trivial.
For instance, if $\chi_2 = 0$ then the $L^\infty$-$L^1$ bound for the first equation
obtained by integrating a suitable linear combination of the first two equations in \eqref{prob}
can be used to obtain certain a priori estimates even for the gradient of the second equation by straightforward semigroup arguments.
However, for \eqref{prob}, bounds for one of the first two equations therein generally do not `automatically' imply bounds for the other one.
As another example, suppose that one could derive $L^\infty$ estimates for both solution components
(ignoring for a moment the fact that these are definitely not easy to obtain): How does one then proceed to obtain, say, Hölder bounds?
At least, classical results for scalar parabolic equations are not applicable.

Therefore, it is not too surprising that the analysis of the system \eqref{prob} with $\chi_1 \gt 0$ and $\chi_2 \gt 0$
is much less developed than for the cases $\chi_1 = 0$ or $\chi_2 = 0$.
To the best of our knowledge, global solutions to \eqref{prob} (with $\chi_1, \chi_2 \gt 0$)
have only been obtained in 1D and only in the weak sense~%
\cite{TaoWinklerExistenceTheoryQualitative2020, TaoWinklerFullyCrossdiffusiveTwocomponent2020}---%
which in turn further indicates the difficulty of the problem \eqref{prob}.

In order to overcome the obstacles outlined above, we thus need to substantially make use of the special structure in \eqref{prob}.
To that end, we carefully design certain functionals in such a way that, in calculating their derivatives, favourable cancellations occur.
We will introduce them in a moment, but before we would like to state our main result.
Making a first step towards extending the knowledge about such systems also in the higher dimensional setting,
we analyze the stability of homogeneous steady states for~\eqref{prob:general} and obtain
\begin{theorem}\label{th:sss}
  Suppose $\Omega \subset \R^n$, $n \in \{1, 2, 3\}$, is a smooth, bounded domain,
  and let
  \begin{align}\label{eq:sss:cond_d_chi}
    D_1, D_2, \chi_1, \chi_2 \gt 0
    \quad \text{and} \quad
    m_1, m_2 \ge 0
  \end{align}

  Suppose either
  \begin{align}\label{eq:sss_h1} \tag{H1}
    \lambda_1 = \lambda_2 = \mu_1 = \mu_2 = a_1 = a_2 = 0    
  \end{align}
  or
  \begin{align}\label{eq:sss_h2} \tag{H2}
    \lambda_1, \lambda_2 \ge 0
    \quad \text{and} \quad
    a_1, a_2, \mu_1, \mu_2 \gt 0.
  \end{align}

  Then there exist $\eps \gt 0$ and $K_1, K_2 \gt 0$ with the following properties:
  For any
  \begin{align}\label{eq:sss:cond_init}
        u_0, v_0
    \in \sobn22 
    \quad \text{
      being nonnegative and,
      if \eqref{eq:sss_h1} holds, with $\textstyle \intom u_0 = m_1$ and $\textstyle \intom v_0 = m_2$},
  \end{align}
  where
  \begin{align}\label{eq:sss:def_sobn22}
          \sobn22 
    \defs \{\,\varphi \in \sob22 \colon \partial_\nu \varphi = 0 \text{ in the sense of traces}\,\},
  \end{align}
  and fulfilling
  \begin{align} \label{eq:sss:cond_eps}
    \|u_0 - \ustar\|_{\sob22} + \|v_0 - \vstar\|_{\sob22} \lt \eps,
  \end{align}
  where
  \begin{align}\label{eq:sss:def_ustar}
    (\ustar, \vstar) \defs 
    \begin{cases}
      \left(\frac{m_1}{|\Omega|}, \frac{m_2}{|\Omega|}\right),
        & \text{if } \eqref{eq:sss_h1} \text{ holds}, \\[1em]
      \left(
        \frac{\lambda_1 \mu_2 + \lambda_2 a_1}{\mu_1 \mu_2 + a_1 a_2},
        \frac{\lambda_2 \mu_1 - \lambda_1 a_2}{\mu_1 \mu_2 + a_1 a_2}
      \right), & \text{if \eqref{eq:sss_h2} holds and } \lambda_2 \mu_1 \gt \lambda_1 a_2, \\[1em]
      \left( \frac{\lambda_1}{\mu_1}, 0 \right),
        & \text{if \eqref{eq:sss_h2} holds and } \lambda_2 \mu_1 \le \lambda_1 a_2,
    \end{cases}
  \end{align}
  there exist a unique pair
  \begin{align*}
    (u, v) \in \left( C^0([0, \infty); \sobn22) \cap C^\infty(\Ombar \times (0, \infty)) \right)^2
  \end{align*}
  solving~\eqref{prob} classically. Moreover, each solution component is nonnegative and $(u, v)$ converges to $(\ustar, \vstar)$ in the sense that
  \begin{align}\label{eq:sss:conv_u_v}
    \|u(\cdot, t) - \ustar\|_{\sob22} + \|v(\cdot, t) - \vstar\|_{\sob22} &\le
    \begin{cases}
      (\frac1{K_1 \eps} + K_2 t)^{-1},  & \text{if \eqref{eq:sss_h2} holds and $\lambda_2 \mu_1 = \lambda_1 a_2$}, \\[1em]
      K_1 \eps \ure^{-K_2 t},           & \text{else}
    \end{cases}
  \end{align}
  for all $t \gt 0$.
\end{theorem}

\begin{remark}
  Let us give some heuristic arguments
  why we believe that the rates in~\eqref{eq:sss:conv_u_v} are, 
  up to the values of $K_1$ and $K_2$ therein,
  optimal.

  For the heat equation, convergence is exponentially fast (take for instance an eigenfunction as initial datum)
  and adding taxis terms (but no terms of zeroth order) should not dramatically speed up the convergence.
  Moreover, in the around $(\ustar, \vstar)$ linearized ODE system,
  $(\ustar, \vstar)$ is a stable fixed point, provided \eqref{eq:sss_h2} with $\lambda_2 \mu_1 \neq \lambda_1 a_2$ holds.
  Hence, also here, `only' an exponential convergence rate can be expected.

  The case~\eqref{eq:sss_h2} with $\lambda_2 \mu_1 = \lambda_1 a_2$ is different.
  As $u$ converges to $\frac{\lambda_1}{\mu_1}$,
  one might expect that $v$ behaves similarly as the solution $\tilde v$ to
  \begin{align*}
      \tilde v'
    = \tilde v \left(\lambda_2 - \mu_2 \tilde v - a_2 \cdot \frac{\lambda_1}{\mu_1} \right)
    = - \mu_2 (\tilde v)^2,
  \end{align*}
  which is given by
  \begin{align*}
    \tilde v(t) = \frac1{\frac1{\tilde v(0)} + \mu_1 t}, \qquad t \ge 0.
  \end{align*}
\end{remark}

\paragraph{Main ideas.}
After obtaining local-in-time solutions by Amann's theory in Lemma~\ref{lm:local_ex},
we will focus our analysis on estimates holding in $\Ombar \times [0, T_\eta)$ for $\eta \gt 0$ to be fixed later,
where $T_\eta \in [0, \infty]$ is the maximal time up to which 
$\|u - \ustar\|_{\leb\infty} + \|v - \vstar\|_{\leb\infty} \lt \eta$.

In the case of~\eqref{eq:sss_h1}, that is, without any cell proliferation, one formally computes
\begin{align*}
  \frac12 \ddt \intom (u - \ustar)^2 + D_1 \intom |\nabla u|^2 = \chi_1 \intom u \nabla u \cdot \nabla v
  \qquad \text{in $(0, \tmax)$}.
\end{align*}
The key idea is that one can rewrite the problematic term on the right hand side as
\begin{align*}
    \chi_1 \intom u \nabla u \cdot \nabla v
  = \chi_1 \intom (u - \ustar) \nabla u \cdot \nabla v  + \chi_1 \ustar \intom \nabla u \cdot \nabla v
  \qquad \text{in $(0, \tmax)$}.
\end{align*}
and note that, as the signs for the taxis terms in \eqref{prob} are opposite, two problematic terms cancel out in calculating
\begin{align*}
  &\pe \ddt \left( \frac{\chi_2 \vstar}{2} \intom (u - \ustar)^2 + \frac{\chi_1 \ustar}{2} \intom (v - \vstar)^2 \right)
       + \chi_2 D_1 \vstar \intom |\nabla u|^2 + \chi_1 D_2 \ustar \intom |\nabla v|^2 \\
  &=   \chi_1 \chi_2 \vstar \intom (u - \ustar) \nabla u \cdot \nabla v
       - \chi_1 \chi_2 \ustar \intom (v - \vstar) \nabla u \cdot \nabla v
  \qquad \text{in $(0, \tmax)$}.
\end{align*}
If $\eta \gt 0$ is chosen small enough, the remaining terms on the right hand side can be absorbed by the dissipative terms---%
at least in $(0, T_\eta)$.

Fortunately, for higher order terms, one can proceed similarly
and thus see that the sum of (norms equivalent to) the $\sob22$~norms of both solution components is decreasing,
which implies $T_\eta = \tmax$, provided $\eta \gt 0$ is small enough
and assuming $T_\eta \gt 0$,
which can be achieved by choosing $\eps \gt 0$ in Theorem~\ref{th:sss} sufficiently small.
Due to the blow-up criterion in Lemma~\ref{lm:local_ex},
one also sees that $\tmax = \infty$.
Convergence to the mean $(\ol u_0, \ol v_0)$ as well as the convergence rate are then merely corollaries of the estimates already gained.

For~\eqref{eq:sss_h2}, however, this idea alone is insufficient.
For instance, if $\ustar \gt 0$ and $\vstar \gt 0$, 
arguing similarly as above,
for any $A_1, A_2 \gt 0$ there is $\eta \gt 0$ such that
\begin{align}\label{eq:intro:ddt_u_2}
  &\pe  \ddt \left( \frac{A_1}{2} \intom (u - \ustar)^2 + \frac{A_2}{2} \intom (v - \vstar)^2 \right) \notag \\
  &\pe  + \frac{A_1 \mu_1}{2} \intom (u - \ustar)^2
        + \frac{A_2 \mu_2}{2} \intom (v - \vstar)^2
        + \frac{A_1 D_1}{2} \intom |\nabla u|^2
        + \frac{A_2 D_2}{2} \intom |\nabla u|^2 \notag \\
  &\le  (A_1 a_1 \ustar - A_2 a_2 \vstar) \intom (u - \ustar) (v - \vstar)
        + (A_1 \chi_1 \ustar - A_2 \chi_2 \vstar) \intom \nabla u \cdot \nabla v 
  \qquad \text{in $(0, T_\eta)$},
\end{align}
see Lemma~\ref{lm:ddt_u_2} and (the proof of) Lemma~\ref{lm:conv_h2_pos_pos}.

For the special case that $(a_1, a_2) = \gamma (\chi_1, \chi_2)$ for some $\gamma \ge 0$,
taking $A_1 \defs \chi_2 \vstar$ and $A_2 \defs  \chi_1 \ustar$
already implies that the right hand side in~\eqref{eq:intro:ddt_u_2} is zero.
Alternatively, if $D_1$ and $D_2$ are sufficiently large compared to $a_1, a_2, \chi_1, \chi_2, \ustar$ and $\vstar$,
the dissipative terms in~\eqref{eq:intro:ddt_u_2} can be used to absorb the terms on the right hand side.
In both these special cases, higher order terms can be handled similarly again so that we can conclude as above.

For arbitrary parameter values,
such shortcuts are apparently unavailable and hence we need to argue differently.
Actually, this is the reason for considering~\eqref{prob} with so many parameters:
We want to emphasize that our approach does not rely on certain relationships between them.

Quite miraculously, appropriately choosing positive linear combinations of the six functionals
\begin{align}\label{eq:intro:functionals}
  \ddt \intom (u - u_\star)^2, \quad
  \ddt \intom (v - v_\star)^2, \quad
  \ddt \intom |\nabla u|^2, \quad
  \ddt \intom |\nabla v|^2, \quad
  \ddt \intom |\Delta u|^2
  \quad \text{and} \quad
  \ddt \intom |\Delta v|^2
\end{align}
still allows for a cancellation of all problematic terms, see Lemma~\ref{lm:conv_h2_pos_pos}.

The remaining case, \eqref{eq:sss_h2} with $\lambda_2 \mu_1 \le \lambda_1 a_2$, is handled in Subsection~\ref{sec:h2_le}.
In a desire to keep the introduction at reasonable length,
we just note here that the proofs also rely on the functionals in~\eqref{eq:intro:functionals},
albeit in a somewhat different fashion as in the first case,
and refer for a more detailed discussion to (the beginning of) Subsection~\ref{sec:h2_le}.
Moreover, the in some sense degenerate case~\eqref{eq:sss_h2} with $\lambda_2 \mu_1 = \lambda_1 a_2$
deserves additional special treatment.
We introduce a new functional in Lemma~\ref{lm:ddt_u}
and discuss directly beforehand why that seems to be necessary.

As a last step, in Lemma~\ref{lm:main_proof} we bring all these estimates together
and prove global existence as well as convergence to $(\ustar, \vstar)$.
Moreover, in Section~\ref{sec:gen}, we discuss possible generalizations of Theorem~\ref{th:sss}.

Finally, in the appendix, we collect certain Gagliardo--Nirenberg-type inequalities used throughout the article.
They might potentially be of independent interest and differentiate themselves from more often seen inequalities in two ways:
Firstly, although we assume $\Omega$ to be bounded, we get rid of the additional additive term on the right hand side.
Secondly, instead of $\|D^2 \varphi\|_{\leb p}$ and $\|D^3 \varphi\|_{\leb p}$,
our version contains only $\|\Delta \varphi\|_{\leb p}$ and $\|\nabla \Delta \varphi\|_{\leb p}$
(for certain values of $p \in (1, \infty)$).

\section{Preliminaries}
\paragraph{Local existence.}
Apparently, trying to prove local existence of classical solutions to \eqref{prob}
by following proofs for systems with a taxis term in just one equation (corresponding to either $\chi_1 = 0$ or $\chi_2 = 0$)
and thus building on the concept of mild solutions and Banach's fixed point theorem
or on Schauder's fixed point theorem
(see for instance \cite{HorstmannWinklerBoundednessVsBlowup2005} or \cite{LankeitLocallyBoundedGlobal2017}, respectively)
is not fruitful---%
at least if we want to consider both arbitrary nonnegative parameters and large initial data.
Therefore, we resort to the abstract existence theory by Amann instead.
\begin{lemma}\label{lm:local_ex}
  Suppose that $\Omega \subset \R^n$, $n \in \N$, is a smooth, bounded domain,
  and let $D_1, D_2, \chi_1, \chi_2 \gt 0$ as well as
  $\lambda_1, \lambda_2, \mu_1, \mu_2, a_1, a_2 \ge 0$.
  Moreover, let $p \gt n$ and $u_0, v_0 \in \sob1p$ be nonnegative.

  Then there exist $\tmax \in (0, \infty]$ and uniquely determined nonnegative
  \begin{align}\label{eq:local_ex:reg}
    u, v \in C^0([0, \tmax); \sob1p) \cap C^\infty(\Ombar \times (0, \tmax))
  \end{align}
  such that $(u, v)$ is a classical solution of~\eqref{prob} and, if $\tmax \lt \infty$, then
  \begin{align}\label{eq:local_ex:ex_crit}
    \limsup_{t \nea \tmax} \left( \|u(\cdot, t)\|_{C^\alpha(\Omega)} + \|v(\cdot, t)\|_{C^\alpha(\Omega)} \right) = \infty
    \qquad \text{for all $\alpha \in (0, 1)$}.
  \end{align}

  Moreover, this solution further satisfies
  \begin{align}\label{eq:local_ex:cont_sob22}
    u, v \in C^0([0, \tmax); \sobn22),
  \end{align}
  provided $u_0, v_0$ satisfy~\eqref{eq:sss:cond_init}.
\end{lemma}
\begin{proof}
  We will construct a solution $U$ to
  \begin{align}\label{eq:local_ex:U_eq}
    \begin{cases}
      U_t = \nabla \cdot (A(U) \nabla U) + F(U), & \text{in $\Omega \times (0, \tmax)$}, \\
      \nu \cdot A(U) \nabla U = 0,               & \text{on $\partial \Omega \times (0, \tmax)$}, \\
      U(\cdot, 0) = U_0,                         & \text{in $\Omega$},
    \end{cases}
  \end{align}
  where
  \begin{align*}
    A\begin{pmatrix}u \\ v\end{pmatrix} \defs
    \begin{pmatrix}
      D_1      & - \chi_1 u \\
      \chi_2 v & D_2
    \end{pmatrix},
    \quad
    F\begin{pmatrix}u \\ v\end{pmatrix} \defs
    \begin{pmatrix}
      u (\lambda_1 - \mu_1 u + a_1 v) \\
      v (\lambda_2 - \mu_2 v - a_2 u)
    \end{pmatrix}
    \quad \text{and} \quad
    U_0 \defs
    \begin{pmatrix}
      u_0 \\
      v_0
    \end{pmatrix}
  \end{align*}
  for $u, v \in \R$.
  Here and below, $\nabla (u, v)^T \defs (\nabla u, \nabla v)^T$, $\nu \cdot (a, b)^T \defs (\nu \cdot a, \nu \cdot b)^T$ etc.\
  for, say, $u, v \in \con1$ and $a, b \in \R^n$.

  If $u, v \ge 0$, then $\tr A((u, v)^T) = D_1 + D_2 \gt 0$ and $\det A((u, v)^T) = D_1 D_2 + \chi_1 \chi_2 u v \gt 0$,
  hence by continuity of the trace and the determinant,
  we may fix an (open) neighborhood $D_0$ of $[0, \infty)^2$ in $\R^2$
  such that the real parts of all eigenvalues of $A((u, v)^T)$ are still positive for all $u, v \in D_0$.
  Thus, defining the operators $\mc A, \mc B$ by
  $\mc A(\eta) U \defs \nabla \cdot (A(\eta) \nabla U)$
  and $\mc B(\eta) \defs \nu \cdot A(\eta) \nabla U$
  for $\eta \in D_0$ and $U \in (\sob2p)^2$,
  we see that $(\mc A(\eta), \mc B(\eta))$ are of separated divergence form
  and hence normally elliptic for all $\eta$ in $D_0$ (cf.\ \cite[Example~4.3(e)]{AmannNonhomogeneousLinearQuasilinear1993}).

  Therefore, we may apply \cite[Theorem~14.4, Theorem~14.6 and Corollary~14.7]{AmannNonhomogeneousLinearQuasilinear1993}
  to obtain $\tmax \gt 0$ and a unique $U \in C^0([0, \tmax); (\sob1p)^2) \cap (C^\infty(\Ombar \times (0, \tmax)))^2$
  solving \eqref{eq:local_ex:U_eq} classically.
  Moreover, since both components of $U$ are nonnegative by the maximum principle (for scalar equations),
  \cite[Theorem~15.3]{AmannNonhomogeneousLinearQuasilinear1993} asserts that in the case of $\tmax \lt \infty$ we have
  \begin{align*}
    \limsup_{t \nea \tmax} \|U(\cdot, t)\|_{(\con\alpha)^2} = \infty \qquad \text{for all $\alpha \in (0, 1)$}.
  \end{align*}
  Thus, $(u, v) \defs U^T$ satisfies the first, second and fourth equations in~\eqref{prob},
  if $\tmax \lt \infty$, then \eqref{eq:local_ex:ex_crit} holds
  and, moreover, $D_1 \partial_\nu u = \chi_1 u \partial_\nu v$ and $D_2 \partial_\nu v = -\chi_2 v \partial_\nu u$ 
  on $\partial \Omega \times (0, \tmax)$.
  As $u$ and $v$ are nonnegative,
  $\partial_\nu u = \frac{\chi_1}{D_1} u \partial_\nu v = -\frac{\chi_1 \chi_2}{D_1 D_2} uv \partial_\nu u$
  on $\partial \Omega \times (0, \tmax)$
  implies $\partial_\nu u \equiv 0$ on $\partial \Omega \times (0, \tmax)$.
  Analogously, we also obtain $\partial_\nu v \equiv 0$ on $\partial \Omega \times (0, \tmax)$,
  hence $(u, v)$ is the unique solution of regularity~\eqref{eq:local_ex:reg} to~\eqref{prob} in $\Ombar \times [0, \tmax)$.

  Since \cite[Theorem~4.1]{AmannNonhomogeneousLinearQuasilinear1993} further asserts
  that, for all $t \in (0, \tmax)$, the operator $\mc A(U(t))$ in $(\leb2)^2$ with $\mc D(\mc A(U(t))) = (\sobn22)^2$
  generates an analytical semigroup on $(\leb2)^2$,
  we may employ \cite[Theorem~10.1]{AmannNonhomogeneousLinearQuasilinear1993}
  to obtain \eqref{eq:local_ex:cont_sob22} for $u_0, v_0 \in \sobn22$.
\end{proof}

\paragraph{Fixing parameters.}
In the sequel, we fix $\Omega \subset \R^n$, $n \in \{1, 2, 3\}$,
parameters as in \eqref{eq:sss:cond_d_chi} and \eqref{eq:sss_h1} or \eqref{eq:sss_h2},
and define $(\ustar, \vstar)$ as in~\eqref{eq:sss:def_ustar}.
Moreover, we set henceforth $\ol \varphi \defs \frac1{|\Omega|} \intom \varphi$ for $\varphi \in \leb1$.

As we will see later in the proofs of Lemma~\ref{lm:phi_abc} and Lemma~\ref{lm:ddt_delta_u_2_h2_pos_zero},
$\sob22$~continuity of both solution components up to $t=0$ will turn out to be crucial.
By Lemma~\ref{lm:local_ex}, this can be achieved if one supposes that $u_0, v_0$ satisfy~\eqref{eq:sss:cond_init}.
Given such initial data, we will denote the solution to~\eqref{prob} constructed in Lemma~\ref{lm:local_ex}
by $(u(u_0, v_0), v(u_0, v_0))$ and its maximal existence time by $\tmax(u_0, v_0)$.
After fixing $(u_0, v_0)$, we will often for the sake of brevity write $(u, v)$ and $\tmax$, respectively, instead.
Also note that all constants below (for instance the $c_i$, $i \in \N$, in several proofs)
depend only on the parameters fixed above, not on $u_0$ and $v_0$.

\paragraph{The functions $f$ and $g$.}
Furthermore, we abbreviate 
\begin{align*}
  f(u, v) \defs u(\lambda_1 - \mu_1 u + a_1 v)
  \quad \text{and} \quad
  g(u, v) \defs v(\lambda_2 - \mu_2 v - a_2 u)
  \qquad \text{for $u, v \gt 0$}.
\end{align*}
Note that $f(\ustar, \vstar) = 0 = g(\ustar \vstar)$ and
\begin{align*}
  \begin{pmatrix}
    f_u(u, v) & f_v(u, v) \\
    g_u(u, v) & g_v(u, v)
  \end{pmatrix}
  = 
  \begin{pmatrix}
    \lambda_1 - 2 \mu_1 u + a_1 v & a_1 u \\
    -a_2 v                        & \lambda_2 - 2 \mu_2 v - a_2 u
  \end{pmatrix}
  \qquad \text{for $u, v \ge 0$},
\end{align*}
that is,
\begin{align*}
    \begin{pmatrix}
      f_u(\ustar, \vstar) & f_v(\ustar, \vstar) \\
      g_u(\ustar, \vstar) & g_v(\ustar, \vstar)
    \end{pmatrix}
  = \begin{cases}
      \begin{pmatrix}
        0 & 0 \\
        0 & 0
      \end{pmatrix},
      & \text{if \eqref{eq:sss_h1} holds}, \\[2em]
      \begin{pmatrix}
        - \mu_1 \ustar & a_1 \ustar \\
        - a_2 \vstar   & -\mu_2 \vstar
      \end{pmatrix},
      & \text{if \eqref{eq:sss_h2} holds and } \lambda_2 \mu_1 \gt \lambda_1 a_2, \\[2em]
      \begin{pmatrix}
        - \lambda_1 & a_1 \ustar \\
        0   & \lambda_2 - \frac{\lambda_1 a_2}{\mu_1}
      \end{pmatrix},
     & \text{if \eqref{eq:sss_h2} holds and } \lambda_2 \mu_1 \le \lambda_1 a_2.
    \end{cases}
\end{align*}
Thus,
\begin{align}\label{eq:fu_gv_ge_0}
  f_u(\ustar, \vstar) \le 0
  \quad \text{as well as} \quad
  g_v(\ustar, \vstar) \le 0
\end{align}
and
\begin{align} \label{eq:fu_gv_gt_0}
  \text{if \eqref{eq:sss_h2} holds and $\lambda_2 \mu_1 \neq \lambda_1 a_2$, then }
  f_u(\ustar, \vstar) \lt 0
  \text{ as well as } 
  g_v(\ustar, \vstar) \lt 0.
\end{align}


\section{Estimates within \tops{$[0, T_\eta)$}{[0, T eta)}}
For $u_0, v_0$ satisfying \eqref{eq:sss:cond_init} and $\eta \gt 0$, set
\begin{align}\label{eq:def_t_eta}
        T_\eta(u_0, v_0)
  \defs \sup \left\{\,
          t \in (0, \tmax(u_0, v_0)): \|u(u_0, v_0) - \ustar\|_{\leb\infty} + \|v(u_0, v_0) - \vstar\|_{\leb\infty} \lt \eta
          \text{ in $(0, t)$}\,\right\}
\end{align}
(with the convention $\sup \emptyset \defs -\infty$).
When confusion seems unlikely, we abbreviate $T_\eta \defs T_\eta(u_0, v_0)$.

In the sequel, we will derive several estimates within $(0, T_\eta)$.
Obviously, if $(0, T_\eta) = \emptyset$, the statements below are trivially true.
Thus upon reading the proofs, the reader might as well always assume that $(0, T_\eta)$ is not empty.
The only exception is Lemma~\ref{lm:main_proof},
where we finally choose $\eps \gt 0$ in \eqref{eq:sss:cond_eps} sufficiently small and guarantee that $T_\eta \gt 0$ for certain $\eta \gt 0$.

Note that $T_{\eta_1} \le T_{\eta_2}$ for $\eta_1 \le \eta_2$.
Moreover,
\begin{align}\label{eq:u_ol_u_eta}
      \|u - \ol u\|_{\leb \infty}
  \le \|u - \ustar\|_{\leb \infty} + \|\ol u - \ustar\|_{\leb \infty}
  =   \|u - \ustar\|_{\leb \infty} + \frac1{|\Omega|} \left| \intom (u - \ustar) \right|
  \le 2 \eta
  \qquad \text{in $(0, T_\eta)$}
\end{align}
and likewise
\begin{align}\label{eq:v_ol_v_eta}
      \|v - \ol v\|_{\leb \infty}
  \le 2 \eta
  \qquad \text{in $(0, T_\eta)$}
\end{align}
for all $\eta \gt 0$,
where $(u, v, \tmax) = (u(u_0, v_0), v(u_0, v_0), \tmax(u_0, v_0))$ for any $u_0, v_0$ complying with~\eqref{eq:sss:cond_init}.

In the remainder of this section,
we derive estimates in $(0, T_\eta)$ for positive linear combinations of
\begin{align}\label{eq:pairs}
  \ddt \intom (u - \ustar)^2
  \quad &\text{and} \quad
  \ddt \intom (v - \vstar)^2, \notag \\
  \ddt \intom |\nabla u|^2
  \quad &\text{and} \quad
  \ddt \intom |\nabla v|^2 \qquad \text{as well as} \qquad \notag \\
  \ddt \intom |\Delta u|^2
  \quad &\text{and} \quad
  \ddt \intom |\Delta v|^2.
\end{align}
We begin by treating the first pair in
\begin{lemma}\label{lm:ddt_u_2_single}
  There is $\eta_0 \gt 0$
  such that if $u_0, v_0$ comply with \eqref{eq:sss:cond_init}
  and $(u, v) = (u(u_0, v_0), v(u_0, v_0))$ denotes the corresponding solution,
  then
  \begin{align}\label{eq:ddt_u_2_single_u}
    &\pe  \frac12 \ddt \intom (u - \ustar)^2
          + \frac{3 D_1}{4} \intom |\nabla u|^2
          +  \left( -f_{u}(\ustar, \vstar) - \eta (a_1 + \mu_1) \right) \intom (u - \ustar)^2  \notag \\
    &\le  a_1 \ustar \intom (u - \ustar) (v - \vstar)
          +  \chi_1 \ustar \intom \nabla u \cdot \nabla v 
          + \frac{ \eta \chi_1}{2} \intom |\nabla v|^2
  \intertext{and}\label{eq:ddt_u_2_single_v}
    &\pe  \frac12 \ddt \intom (v - \vstar)^2
          + \frac{3 D_2}{4} \intom |\nabla v|^2
          +  \left( -g_{v}(\ustar, \vstar) - \eta (a_2 + \mu_2) \right) \intom (v - \vstar)^2  \notag \\
    &\le  -  a_2 \vstar \intom (u - \ustar) (v - \vstar)
          -  \chi_2 \ustar \intom \nabla u \cdot \nabla v 
          + \frac{ \eta \chi_2}{2} \intom |\nabla u|^2
  \end{align}
  hold in $(0, T_\eta)$ for all $\eta \in (0, \eta_0)$, where $T_\eta$ is given by \eqref{eq:def_t_eta}.
\end{lemma}
\begin{proof}
  Let
  \begin{align}\label{eq:ddt_u_2:eta_0}
    \eta_0 \defs \frac12 \min\left\{\frac{D_1}{\chi_1}, \frac{D_2}{\chi_2} \right\}.
  \end{align}

  Fixing $u_0, v_0$ satisfying with~\eqref{eq:sss:cond_init},
  by a direct calculation, we see that
  \begin{align*}
        \frac12 \ddt \intom (u - \ustar)^2
        +  D_1 \intom |\nabla u|^2
    =   \chi_1 \intom u \nabla u \cdot \nabla v
        + \intom f(u, v) (u-\ustar)
  \end{align*}
  holds in $(0, \tmax)$.

  For any $\eta \gt 0$,
  we have therein by Young's inequality
  \begin{align*}
          \chi_1 \intom u \nabla u \cdot \nabla v
    &=    \chi_1 \ustar \intom \nabla u \cdot \nabla v
          + \chi_1 \intom (u - u _\star) \nabla u \cdot \nabla v \\
    &\le  \chi_1 \ustar \intom \nabla u \cdot \nabla v
          + \frac{\eta \chi_1}{2} \intom |\nabla u|^2
          + \frac{\eta \chi_1}{2} \intom |\nabla v|^2
    \qquad \text{in $(0, T_\eta)$}.
  \end{align*}

  Moreover, as $f(\ustar, \vstar) = 0$,
  \begin{align*}
          \intom f(u, v) (u - \ustar)
    &=    \intom f(u, \vstar) (u - \ustar)
          + a_1 \intom  u(v - \vstar) (u - \ustar) \\
    &=    f_u(\ustar, \vstar) \intom (u - \ustar)^2
          + \frac{f_{uu}(\ustar, \vstar)}{2} \intom (u - \ustar)^3 \\
    &\pe  + a_1 \intom (u - \ustar)^2 (v - \vstar)
          + a_1 \ustar \intom (u - \ustar) (v - \vstar)
    \qquad \text{in $(0, \tmax)$}.
  \end{align*}

  Since $f_{uu}(\ustar, \vstar) = -2\mu_1$,
  we may further estimate
  \begin{align*}
         \frac{f_{uu}(\ustar, \vstar)}{2} \intom (u - \ustar)^3
    \le \eta \mu_1 \intom (u - \ustar)^2
    \qquad \text{in $(0, T_\eta)$ for all $\eta \gt 0$}
  \end{align*}
  and
  \begin{align*}
         a_1 \intom (u - \ustar)^2 (v - \vstar)
    \le \eta a_1 \intom (u - \ustar)^2
    \qquad \text{in $(0, T_\eta)$ for all $\eta \gt 0$}.
  \end{align*}

  Noting that \eqref{eq:ddt_u_2:eta_0} implies $D_1 - \frac{\eta_0 \chi_1}{2} \ge \frac34 D_1$,
  we may combine these estimates to obtain \eqref{eq:ddt_u_2_single_u},
  while \eqref{eq:ddt_u_2_single_v} follows from an analogous computation.
\end{proof}

For sufficiently small $\eta$ and suitable linear combinations of \eqref{eq:ddt_u_2_single_u} and \eqref{eq:ddt_u_2_single_v},
the terms $\frac{\eta \chi_1}{2} \intom |\nabla v|^2$ and $\frac{\eta \chi_2}{2} \intom |\nabla u|^2$
can be absorbed by the dissipative terms therein.
\begin{lemma}\label{lm:ddt_u_2}
  For any $A_1, A_2 \gt 0$, there is $\eta_0 \gt 0$ such that
  whenever $u_0, v_0$ satisfy~\eqref{eq:sss:cond_init},
  then the corresponding solution $(u, v) = (u(u_0, v_0), v(u_0, v_0))$ satisfies
  \begin{align}\label{eq:ddt_u_2:statement}
    &\pe  \ddt \left( \frac{A_1}{2} \intom (u - \ustar)^2 + \frac{A_2}{2} \intom (v - \vstar)^2 \right)
          + \frac{A_1 D_1}{2} \intom |\nabla u|^2
          + \frac{A_2 D_2}{2} \intom |\nabla v|^2 \notag \\
    &\pe  + A_1 \left( -f_{u}(\ustar, \vstar) - \eta (a_1 + \mu_1) \right) \intom (u - \ustar)^2 
          + A_2 \left( -g_{v}(\ustar, \vstar) - \eta (a_2 + \mu_2) \right) \intom (v - \vstar)^2  \notag \\
    &\le  (A_1 a_1 \ustar - A_2 a_2 \vstar) \intom (u - \ustar) (v - \vstar)
          + (A_1 \chi_1 \ustar - A_2 \chi_2 \vstar) \intom \nabla u \cdot \nabla v 
    \qquad \text{in $(0, T_\eta)$}
  \end{align}
  for all $\eta \lt \eta_0$, where $T_\eta$ is as in~\eqref{eq:def_t_eta}.
\end{lemma}
\begin{proof}
  Lemma~\ref{lm:ddt_u_2_single} allows us
  to choose $\eta_1$ such that \eqref{eq:ddt_u_2_single_u} and \eqref{eq:ddt_u_2_single_v} hold in $(0, T_{\eta_1})$.
  Let moreover  $A_1, A_2 \gt 0$,
  fix $\eta_2 \gt 0$ sufficiently small such that
  \begin{align*}
    \frac{A_2 \eta_2 \chi_2}{2} \le \frac{A_1 D_1}{4}
    \quad \text{and} \quad
    \frac{A_1 \eta_2 \chi_1}{2} \le \frac{A_2 D_2}{4}
  \end{align*}
  and set $\eta_0 \defs \min\{\eta_0, \eta_1\}$.
   
  The statement then immediately follows upon multiplying \eqref{eq:ddt_u_2_single_u} and \eqref{eq:ddt_u_2_single_v}
  with $A_1$ and $A_2$, respectively, and adding these inequalities together.
\end{proof}

Next, we handle the second pair in \eqref{eq:pairs}, this time only in a coupled version.
\begin{lemma}\label{lm:ddt_nabla_u_2}
  Let $B_1, B_2 \gt 0$.
  There is $\eta \gt 0$ such that for any $u_0, v_0$ complying with~\eqref{eq:sss:cond_init}
  we have
  \begin{align*}
    &\pe  \ddt \left( \frac{B_1}{2} \intom |\nabla u|^2 + \frac{B_2}{2} \intom |\nabla v|^2 \right)
          + \frac{B_1D_1}{2} \intom |\Delta u|^2
          + \frac{B_2D_2}{2} \intom |\Delta v|^2 \\
    &\le  (B_1 a_1 \ustar - B_2 a_2 \vstar) \intom \nabla u \cdot \nabla v
          + (B_1 \chi_1 \ustar - B_2 \chi_2 \vstar) \intom \Delta u \Delta v
    \qquad \text{in $(0, T_\eta)$},
  \end{align*}
  where again $(u, v) \defs (u(u_0, v_0), v(u_0, v_0))$ and $T_\eta \defs T_\eta(u_0, v_0))$ is given by~\eqref{eq:def_t_eta}.
\end{lemma}
\begin{proof}
  Let $B_1, B_2 \gt 0$.
  We begin by fixing some parameters:
  By the Gagliardo--Nirenberg inequality~\ref{lm:gni_sob_22},
  there is $c_1 \gt 0$ such that
  \begin{align}\label{eq:ddt_nabla_u_2:gni}
    \intom |\nabla \varphi|^4 \le c_1 \|\varphi - \ol \varphi\|_{\leb \infty}^2 \intom |\Delta \varphi|^2
    \qquad \text{for all $\varphi \in \con2$ with $\partial_\nu \varphi = 0$ on $\partial \Omega$}.
  \end{align}
  Choose $\eta \gt 0$ so small that
  \begin{align*}
            M_1(\eta)
    &\defs  \frac{B_1 \eta \chi_1}{2}
            + \frac{B_2 \eta \chi_2}{2}
            + \frac{2 B_1 \eta^2 \chi_1^2 c_1}{D_1}
            + \frac{2 B_2 \eta^2 \chi_2^2 c_1}{D_2}
            + B_1 \cp \eta (2\mu_1 + a_1)
            + \frac{B_1 \cp a_1 \eta}{2}
            + \frac{B_2 \cp a_2 \eta}{2}
    \intertext{and}
            M_2(\eta)
    &\defs  \frac{B_1 \eta \chi_1}{2}
            + \frac{B_2 \eta \chi_2}{2}
            + \frac{2 B_1 \eta^2 \chi_1^2 c_1}{D_1}
            + \frac{2 B_2 \eta^2 \chi_2^2 c_1}{D_2}
            + B_2 \cp \eta (2\mu_2 + a_2)
            + \frac{B_1 \cp a_1 \eta}{2}
            + \frac{B_2 \cp a_2 \eta}{2},
  \end{align*}
  where $\cp$ is as in Lemma~\ref{lm:poincare},
  fulfill
  \begin{align}\label{eq:ddt_nabla_u_2:mi_small}
    M_1(\eta) \lt \frac{B_1 D_1}{4}
    \quad \text{and} \quad
    M_2(\eta) \lt \frac{B_2 D_2}{4}.
  \end{align}

  Fixing $u_0, v_0$ as in \eqref{eq:sss:cond_init}, we calculate
  \begin{align*}
        \frac12 \ddt \intom |\nabla u|^2 + D_1 \intom |\Delta u|^2
    &=  \chi_1 \intom u \Delta u \Delta v
        + \chi_1 \intom \nabla u \cdot \nabla v \Delta u
        + \intom f_u(u, v) |\nabla u|^2
        + a_1 \intom u \nabla u \cdot \nabla v \\
    &\sfed I_1 + I_2 + I_3 + I_4
    \qquad \text{in $(0, \tmax)$}.
  \end{align*}

  Therein is
  \begin{align*}
          I_1
    &=    \chi_1 \ustar \intom \Delta u \Delta v
          + \chi_1 \intom (u - \ustar) \Delta u \Delta v \\
    &\le  \chi_1 \ustar \intom \Delta u \Delta v
          + \frac{\eta \chi_1}{2} \intom |\Delta u|^2
          + \frac{\eta \chi_1}{2} \intom |\Delta v|^2
    \qquad \text{in $(0, T_\eta)$.}
  \end{align*}

  Furthermore, by \eqref{eq:ddt_nabla_u_2:gni}, \eqref{eq:u_ol_u_eta} and Young's inequality,
  \begin{align*}
          I_2
    &\le  \frac{D_1}{4} \intom |\Delta u|^2 + \frac{\chi_1^2}{D_1} \intom |\nabla u|^2 |\nabla v|^2 \\
    &\le  \frac{D_1}{4} \intom |\Delta u|^2
          + \frac{\chi_1^2}{2D_1} \intom |\nabla u|^4
          + \frac{\chi_1^2}{2D_1} \intom |\nabla v|^4 \\
    &\le  \frac{D_1}{4} \intom |\Delta u|^2
          + \frac{2 \eta^2 \chi_1^2 c_1}{D_1} \intom |\Delta u|^2
          + \frac{2 \eta^2 \chi_1^2 c_1}{D_1} \intom |\Delta v|^2
    \qquad \text{in $(0, T_\eta)$.}
  \end{align*}

  Moreover, due to~\eqref{eq:fu_gv_ge_0}, by the mean value theorem, as $f_{uu} \equiv 2\mu_1$ and $f_{uv} \equiv a_1$
  and because of the Poincar\'e inequality~\ref{lm:poincare} (with $\cp \gt 0$ as in that lemma),
  \begin{align*}
          I_3
    &\le  \intom (f_u(u, v) - f_u(\ustar, \vstar)) |\nabla u|^2 \\
    &\le  \intom \left(
            \|f_{uu}\|_{L^\infty((0, \infty)^2)} |u - \ustar|
            + \|f_{uv}\|_{L^\infty((0, \infty)^2)} |v - \vstar|
          \right) |\nabla u|^2 \\
    &\le  \eta (2\mu_1 + a_1) \cp \intom |\Delta u|^2
    \qquad \text{in $(0, T_\eta)$.}
  \end{align*}

  Finally, by Young's inequality and the Poincar\'e inequality~\ref{lm:poincare} (again with $\cp \gt 0$ as in that lemma),
  \begin{align*}
          I_4
    &=    a_1 \ustar \intom \nabla u \cdot \nabla v
          + a_1 \intom (u - \ustar) \nabla u \cdot \nabla v \\
    &\le  a_1 \ustar \intom \nabla u \cdot \nabla v
          + \frac{\eta a_1 \cp}{2} \left( \intom |\Delta u|^2 + \intom |\Delta v|^2 \right)
    \qquad \text{in $(0, T_\eta)$.}
  \end{align*}

  Along with an analogous computation for $v$,
  these estimates imply
  \begin{align*}
    &\pe  \ddt \left( \frac{B_1}{2} \intom |\nabla u|^2 + \frac{B_{2}}{2} \intom |\nabla v|^2 \right) \\
    &\pe  + \left( \frac{3 B_1 D_1}{4} - M_1(\eta) \right) \intom |\Delta u|^2
          + \left( \frac{3 B_2 D_2}{4} - M_2(\eta) \right) \intom |\Delta v|^2 \\
    &\le  (B_1 a_1 \ustar - B_2 a_2 \vstar) \intom \nabla u \cdot \nabla v
          + (B_1 \chi_1 \ustar - B_2 \chi_2 \vstar) \intom \Delta u \Delta v
    \qquad \text{in $(0, T_\eta)$}.
  \end{align*}
  The statement follows due to \eqref{eq:ddt_nabla_u_2:mi_small}.
\end{proof}

At last, we deal with the third pair in \eqref{eq:pairs}.
\begin{lemma}\label{lm:ddt_delta_u_2}
  For any $C_1, C_2 \gt 0$, there exists $\eta \gt 0$ such that  
  $(u, v, T_\eta) \defs (u(u_0, v_0), v(u_0, v_0), T_\eta(u_0, v_0))$, where $T_\eta$ is defined in \eqref{eq:def_t_eta}, satisfies
  \begin{align*}
    &\pe  \ddt \left( \frac{C_1}{2} \intom |\Delta u|^2 + \frac{C_2}{2} \intom |\Delta v|^2 \right)
          + \frac{C_1 D_1}{2} \intom |\nabla \Delta u|^2 + \frac{C_1 D_2}{2} \intom |\nabla \Delta v|^2 \\
    &\le  (C_1 a_1 \ustar - C_2 a_2 \vstar) \intom \Delta u \Delta v
          + (C_1 \chi_1 \ustar - C_2 \chi_2 \vstar) \intom \nabla \Delta u \cdot \nabla \Delta v
    \qquad \text{in $(0, T_\eta)$},
  \end{align*}
  provided $u_0, v_0$ fulfill~\eqref{eq:sss:cond_init}.
\end{lemma}
\begin{proof}
  Fix $C_1, C_2 \gt 0$.
  Let us again begin by fixing some constants:
  By Lemma~\ref{lm:gni_sob_32} and Lemma~\ref{lm:w22_delta_2}, there is $c_1 \gt 0$ such that
  \begin{align}\label{eq:ddt_delta_u_2:gni_1}
    6 \max\left\{\frac{\chi_1^2}{D_1}, \frac{\chi_2^2}{D_2}\right\}
      \intom |\nabla \varphi|^6 &\le c_1 \|\varphi - \ol \varphi\|_{\leb\infty}^4 \intom |\nabla \Delta \varphi|^2
    \qquad\text{for all $\varphi \in \con3$ with $\partial_\nu \varphi = 0$ on $\partial \Omega$}
  \intertext{as well as}\label{eq:ddt_delta_u_2:gni_2}
    12 \max\left\{\frac{\chi_1^2}{D_1}, \frac{\chi_2^2}{D_2}\right\}
      \intom |D^2 \varphi|^3 &\le c_1 \|\varphi - \ol \varphi\|_{\leb\infty} \intom |\nabla \Delta \varphi|^2
    \qquad\text{for all $\varphi \in \con3$ with $\partial_\nu \varphi = 0$ on $\partial \Omega$}
  \end{align}
  and Lemma~\ref{lm:gni_sob_22} provides us with $c_2 \ge 1$ such that
  \begin{align}\label{eq:ddt_delta_u_2:gni_3}
        \intom |\nabla \varphi|^4
    \le c_2 \|\varphi - \ol \varphi\|_{\leb \infty}^2 \intom |\Delta \varphi|^2
    \qquad \text{for all $\varphi \in \con2$ with $\partial_\nu \varphi = 0$ on $\partial \Omega$}.
  \end{align}
  Fix furthermore $\cp$ as in Lemma~\ref{lm:poincare} and choose $\eta \gt 0$ so small that
  \begin{align*}
            M_1(\eta)
    &\defs  \frac{C_1 \eta \chi_1}{2} + \frac{C_2 \eta \chi_2}{2}
            + (C_1 + C_2) c_1 (2\eta + 16\eta^4)
            + \frac{C_1 \cp c_2 \eta (9a_1 + 14\mu_1)}{2}
            + \frac{5C_2 \cp a_2 c_2 \eta}{2}
    \intertext{and}
            M_2(\eta)
    &\defs  \frac{C_1 \eta \chi_1}{2} + \frac{C_2 \eta \chi_2}{2}
            + (C_1 + C_2) c_1 (2\eta + 16\eta^4)
            + \frac{C_2 \cp c_2 \eta (9a_2 + 14\mu_2)}{2}
            + \frac{5C_1 \cp a_1 c_2 \eta}{2}
  \end{align*}
  satisfy
  \begin{align}\label{eq:ddt_delta_u_2:mi_small}
    M_1(\eta) \lt \frac{C_1 D_1}{4} 
    \quad \text{and} \quad
    M_2(\eta) \lt \frac{C_2 D_1}{4}.
  \end{align}
  
  Fix also $u_0, v_0$ complying with \eqref{eq:sss:cond_init}.
  Since $\partial_\nu u = 0$ on $\partial \Omega \times (0, \tmax)$
  implies $(\partial_\nu u)_t = 0$ on $\partial \Omega \times (0, \tmax)$
  and as $|\Delta \varphi| \le \sqrt n |D^2 \varphi|$ for all $\varphi \in C^2(\Ombar)$,
  we may calculate
  \begin{align}\label{eq:ddt_delta_u_2:ddt}
    &\pe  \frac12 \ddt \intom |\Delta u|^2 \notag \\
    &=    - \intom \nabla u_t \cdot \nabla \Delta u + \int_{\partial \Omega} (\partial_\nu u)_t \Delta u \notag \\
    &=    - D_1 \intom |\nabla \Delta u|^2
          + \chi_1 \intom \nabla (u \Delta v + \nabla u \cdot \nabla v) \cdot \nabla \Delta u
          - \intom \nabla (f(u, v)) \cdot \nabla \Delta u \notag \\
    &\le  - D_1 \intom |\nabla \Delta u|^2
          - \intom \nabla (f(u, v)) \cdot \nabla \Delta u \notag \\
    &\pe  + \chi_1 \intom u \nabla \Delta u \cdot \nabla \Delta v
          + \chi_1 \intom (|D^2 u| |\nabla v| + (1 + \sqrt n) |D^2 v| |\nabla u|) |\nabla \Delta u|
    \qquad \text{in $(0, \tmax)$}.
  \end{align}

  Therein is by Young's inequality
  \begin{align*}
        \chi_1 \intom u \nabla \Delta v \cdot \nabla \Delta u
    &=  \chi_1 \ustar \intom \nabla \Delta v \cdot \nabla \Delta u
        + \chi_1 \intom (u - \ustar) \nabla \Delta v \cdot \nabla \Delta u \\
    &=  \chi_1 \ustar \intom \nabla \Delta v \cdot \nabla \Delta u
        + \frac{\eta \chi_1}{2} \intom |\nabla \Delta u|^2
        + \frac{\eta \chi_1}{2} \intom |\nabla \Delta v|^2
    \qquad \text{in $(0, T_\eta)$}.
  \end{align*}
  
  Again by Young's inequality
  combined with $\sqrt n \le 2$, \eqref{eq:ddt_delta_u_2:gni_1}, \eqref{eq:ddt_delta_u_2:gni_2}, \eqref{eq:u_ol_u_eta} and \eqref{eq:v_ol_v_eta},
  we further estimate
  \begin{align*}
    &\pe  \chi_1 \intom \left( |D^2 u| |\nabla v| + (1 + \sqrt n) |D^2 v| |\nabla u| \right) |\nabla \Delta u| \\
    &\le  \frac{D_1}{4} \intom |\nabla \Delta u|^2
          + \frac{2\chi_1^2}{D_1} \intom |D^2 u|^2 |\nabla v|^2
          + \frac{18\chi_1^2}{D_1} \intom |D^2 v|^2 |\nabla u|^2 \\
    &\le  \frac{D_1}{4} \intom |\nabla \Delta u|^2
          + \frac{4 \chi_1^2}{3 D_1} \intom |D^2 u|^3
          + \frac{2 \chi_1^2}{3 D_1} \intom |\nabla v|^6
          + \frac{12 \chi_1^2}{D_1} \intom |D^2 v|^3
          + \frac{6 \chi_1^2}{D_1} \intom |\nabla u|^6\\
    &\le  \left(\frac{D_1}{4} + 2 c_1 \eta + 16 c_1 \eta^4 \right) \intom |\nabla \Delta u|^2
          + \left(2 c_1 \eta + 16 c_1 \eta^4 \right) \intom |\nabla \Delta v|^2
    \qquad \text{in $(0, T_\eta)$}.
  \end{align*}
  (We note that we estimated $\sqrt n \le 2$ only to keep the expressions as simple as possible.
  After possibly enlarging certain constants, the same estimates also holds in the higher dimensional settings;
  that is, no dimension restriction is imposed here.)

  Regarding the remaining term in \eqref{eq:ddt_delta_u_2:ddt},
  we first note that 
  \begin{align*}
    D^2 f(u, v) =
    \begin{pmatrix}
      - 2 \mu_1 & a_1 \\
      a_1       & 0
    \end{pmatrix}
    \qquad \text{in $(0, \tmax)$}
  \end{align*}
  and that \eqref{eq:fu_gv_ge_0} implies
  \begin{align*}
          f_u(u, v) 
    &=    f_u(u, \vstar) + a_1 (v - \vstar) \\
    &=    f_u(\ustar, \vstar) + f_{uu}(\ustar, \vstar) (u - \ustar) + a_1 (v - \vstar) \\
    &\le  -2\mu_1 (u - \ustar) + a_1 (v - \vstar)
    \qquad \text{in $(0, \tmax)$}.
  \end{align*}
  Therefore, an integration by parts and applications of  Young's inequality as well as Poincar\'e's inequality~\ref{lm:poincare} yield
  \begin{align*}
    &\pe  - \intom \nabla(f(u, v)) \cdot \nabla \Delta u \\
    &=    - \intom f_u(u, v) \nabla u \cdot \nabla \Delta u
          - \intom f_v(u, v) \nabla v \cdot \nabla \Delta u \\
    &=    \intom f_u(u, v) |\Delta u|^2
          + \intom f_{uu}(u, v) |\nabla u|^2 \Delta u
          + 2\intom f_{uv}(u, v) \nabla u \cdot \nabla v \Delta u \\
    &\pe  + \intom f_v(u, v) \Delta u \Delta v
          + \intom f_{vv}(u, v) |\nabla v|^2 \Delta u \\
    &\le  \eta (a_1 + 2\mu_1) \intom |\Delta u|^2
          + 2 \mu_1 \intom |\nabla u|^2 |\Delta u|
          + 2 a_1 \intom \nabla u \cdot \nabla v \Delta u \\
    &\pe  + a_1 \intom (u - \ustar) \Delta u \Delta v
          + a_1 \ustar \intom \Delta u \Delta v \\
    &\le  \cp \eta (a_1 + 2\mu_1) \intom |\nabla \Delta u|^2
          + a_1 \ustar \intom \Delta u \Delta v \\
    &\pe  + \eta \mu_1 \intom |\Delta u|^2 + \frac{\mu_1}{\eta} \intom |\nabla u|^4 \\
    &\pe  + a_1 \eta \intom |\Delta u|^2 + \frac{a_1}{2\eta} \intom |\nabla u|^4 + \frac{a_1}{2\eta} \intom |\nabla v|^4 \\
    &\pe  + \frac{a_1 \eta}{2} \intom |\Delta u|^2 + \frac{a_1 \eta}{2} \intom |\Delta v|^2 \\
    &\le  \frac{\cp \eta (5 a_1 + 6 \mu_1)}{2} \intom |\nabla \Delta u|^2
          + \frac{\cp a_1 \eta}{2} \intom |\nabla \Delta v|^2
          + a_1 \ustar \intom \Delta u \Delta v \\
    &\pe  + \frac{2 \mu_1 + a_1}{2\eta} \intom |\nabla u|^4
          + \frac{a_1}{2\eta} \intom |\nabla v|^4
    \qquad\text{in $(0, T_\eta$)}.
  \end{align*}
  Herein we make use of \eqref{eq:ddt_delta_u_2:gni_3}, \eqref{eq:u_ol_u_eta}
  and Poincar\'e's inequality~\ref{lm:poincare}
  to further conclude
  \begin{align*}
        \intom |\nabla u|^4
    \le c_2 \|u - \ol u\|_{\leb \infty}^2 \intom |\Delta u|^2
    \le 4 \cp c_2 \eta^2 \intom |\nabla \Delta u|^2
    \qquad \text{in $(0, T_\eta)$}
  \end{align*}
  and, likewise, now using \eqref{eq:v_ol_v_eta} instead of \eqref{eq:u_ol_u_eta},
  \begin{align*}
        \intom |\nabla v|^4
    \le 4 \cp c_2 \eta^2 \intom |\nabla \Delta v|^2
    \qquad \text{in $(0, T_\eta)$}.
  \end{align*}
  Thus, due to $c_2 \ge 1$,
  \begin{align*}
          - \intom \nabla(f(u, v)) \cdot \nabla \Delta u
    &\le  \frac{\cp c_2 \eta (9 a_1 + 14 \mu_1)}{2} \intom |\nabla \Delta u|^2
          + \frac{5 \cp a_1 c_2 \eta}{2} \intom |\nabla \Delta v|^2
          + a_1 \ustar \intom \Delta u \Delta v
  \end{align*}
  holds in $(0, T_\eta)$.

  As usual, we now combine the estimates above with analogous computations for $v$ to obtain
  \begin{align*}
    &\pe  \ddt \left( \frac{C_1}{2} \intom |\Delta u|^2 + \frac{C_2}{2} \intom |\Delta v|^2 \right) \\
    &\pe  + \left( \frac{3C_1 D_1}{4} - M_1(\eta) \right) \intom |\nabla \Delta u|^2
          + \left( \frac{3C_2 D_2}{4} - M_2(\eta) \right) \intom |\nabla \Delta v|^2 \\
    &\le  (C_1 a_1 \ustar - C_2 a_2 \vstar) \intom \Delta u \Delta v
          + (C_1 \chi_1 \ustar - C_2 \chi_2 \vstar) \intom \nabla \Delta u \cdot \nabla \Delta v
    \qquad \text{in $(0, T_\eta)$},
  \end{align*}
  which in virtue of~\eqref{eq:ddt_delta_u_2:mi_small} implies the statement.
\end{proof}

\section{Deriving \tops{$\sob22$}{W22} bounds for \tops{$u$}{u} and \tops{$v$}{v}}\label{sec:w22_est}
In this section, we will make use of the estimates gained in the previous section
to finally obtain $\sob22$~bounds for both solution components.
That is, we will aim to bound $\|u - \ustar\|_{\sob22} + \|v - \vstar\|_{\sob22}$ by, say, $\frac{\eta}{2}$ in $(0, T_\eta)$
(for a certain $\eta \gt 0$),
as then $T_\eta = \tmax = \infty$ can be concluded---%
provided $T_\eta \gt 0$ which in turn can be achieved by requiring $\|u_0 - \ustar\|_{\sob22} + \|v_0 - \vstar\|_{\sob22}$ to be sufficiently small.

In the sequel, we distinguish between multiple cases.
More concretely, we will handle
\begin{itemize}
  \item \eqref{eq:sss_h1} in Lemma~\ref{lm:conv_h1}, 
  \item \eqref{eq:sss_h2} with $\lambda_2 \mu_1 \gt \lambda_1 a_2$ in Lemma~\ref{lm:conv_h2_pos_pos}, 
  \item \eqref{eq:sss_h2} with $\lambda_2 \mu_1 \lt \lambda_1 a_2$
        in Lemma~\ref{lm:ddt_delta_u_2_h2_pos_zero} and Lemma~\ref{lm:conv_h2_pos_zero}
  \item \eqref{eq:sss_h2} with $\lambda_2 \mu_1 = \lambda_1 a_2$ and $\lambda_1 \gt 0$
        in Lemma~\ref{lm:conv_h2_zero_zero_1_v}~(ii) and Lemma~\ref{lm:conv_h2_zero_zero_1_u}
        as well as
  \item \eqref{eq:sss_h2} with $\lambda_1 = \lambda_2 = 0$
        in Lemma~\ref{lm:conv_h2_zero_zero_2}.
\end{itemize}
These five cases can be divided into two groups, the first of which we deal with in the following subsection.

\subsection{The cases \tops{\eqref{eq:sss_h1}}{(H1)} and \tops{\eqref{eq:sss_h2}}{H2} with \tops{$\lambda_2 \mu_1 \gt \lambda_1 a_2$}{lambda2 mu1 > lambda1 a2}}
If either \eqref{eq:sss_h1} holds with $m_1, m_2 \gt 0$ or \eqref{eq:sss_h2} holds with $\lambda_2 \mu_1 \gt \lambda_1 a_2$,
$\ustar$ and $\vstar$ are positive---%
which is the reason these cases can be handled in a similar fashion.
In both cases, we will aim to apply the following elementary lemma.

\begin{lemma}\label{lm:phi_abc}
  For $A, B, C \gt 0$ and $\varphi \in \sob22$ set
  \begin{align}\label{eq:phi_abc:def_phi}
          \phi_{A, B, C}(\varphi)
   \defs  \frac{A}{2} \intom \varphi^2
          + \frac{B}{2} \intom |\nabla \varphi|^2
          + \frac{C}{2} \intom |\Delta \varphi|^2
  \end{align}
  and let $A_1, A_2, B_1, B_2, C_1, C_2 \gt 0$, $\eta \gt 0$ and $K_2 \gt 0$.

  There is $K_1 \gt 0$ such that,
  if $u_0, v_0$ comply with \eqref{eq:sss:cond_init}, $T_\eta$ is as in \eqref{eq:def_t_eta} and
  \begin{align}\label{eq:phi_abc:def_y}
    y \colon [0, T_\eta) \ra \R,
    \quad t \mapsto \phi_{A_1, B_1, C_1}(u(\cdot, t) - \ustar) + \phi_{A_2, B_2, C_2}(v(\cdot, t) - \vstar)
  \end{align}
  fulfills
  \begin{align}\label{eq:phi_abc:cond}
    y'(t) \le -2K_y(t)
    \qquad \text{in $(0, T_\eta)$},
  \end{align}
  then
  \begin{align}\label{eq:conv_exp}
        \|u(\cdot, t) - \ustar\|_{\sob22} + \|v(\cdot, t) - \vstar\|_{\sob22} 
    \le K_1 \ure^{-K_2 t} \left( \|u_0 - \ustar\|_{\sob22} + \|v_0 - \vstar\|_{\sob22} \right)
  \end{align}
  for all $t \in (0, T_\eta)$.
\end{lemma}
\begin{proof}
  As $\sob22$ continuity of $u$ and $v$ up to $t=0$ is ensured by~\eqref{eq:local_ex:cont_sob22},
  we may make use of an ODE comparison argument to obtain
  \begin{align*}
          y(t)
    &\le  \ure^{-2K_2 t} y(0)
    \qquad \text{for all $t \in (0, T_\eta)$}.
  \end{align*}
  The statement follows by taking square roots on both sides
  and noting that $\|\varphi\| \defs \sqrt{\phi_{A, B, C}(\varphi)}$ defines for $A, B, C \gt 0$
  a norm on $\sobn22$, which is equivalent to the usual one by Lemma~\ref{lm:w22_delta_2}.
\end{proof}

For both cases covered in this subsection,
we will now choose $A_1, A_2, B_1, B_2, C_1, C_2 \gt 0$ appropriately 
so that Lemma~\ref{lm:phi_abc} is applicable.

\begin{lemma}\label{lm:conv_h1}
  Suppose~\eqref{eq:sss_h1}.
  Then there are $\eta \gt 0$ and $K_1, K_2 \gt 0$ such that~\eqref{eq:conv_exp} holds
  for all $u_0, v_0$ satisfying~\eqref{eq:sss:cond_init}.
\end{lemma}
\begin{proof}
  In the case of \eqref{eq:sss_h1} with $m_1 = 0$ or $m_2 = 0$, that is, if at least one of the initial data is trivial,
  the uniqueness statement in Lemma~\ref{lm:local_ex} asserts
  that one solution component is constantly zero while the other solves the heat equation.
  As in that case the statement becomes trivial, we may assume $m_1 \gt 0$ and $m_2 \gt 0$.

  Then $\ustar, \vstar \gt 0$ and hence
  $A_1 = B_1 = C_1 \defs \chi_2 \vstar$ as well as $A_2 = B_2 = C_2 \defs \chi_1 \ustar$ are positive as well.
  Because of
  \begin{align*}
    A_1 \chi_1 \ustar - A_2 \chi_2 \vstar = 0, \quad
    B_1 \chi_1 \ustar - B_2 \chi_2 \vstar = 0, \quad
    C_1 \chi_1 \ustar - C_2 \chi_2 \vstar = 0
  \end{align*}
  and \eqref{eq:sss_h1}, Lemma~\ref{lm:ddt_u_2}, Lemma~\ref{lm:ddt_nabla_u_2} and Lemma~\ref{lm:ddt_delta_u_2} assert
  that there is $\eta \gt 0$ such that
  \begin{align*}
    &\pe  \ddt \Big(
            \phi_{A_1, B_1, C_1}(u(\cdot, t) - \ustar)
            + \phi_{A_2, B_2, C_2}(v(\cdot, t) - \vstar)
          \Big)
          + \frac{C_1 D_1}{2} \intom |\nabla \Delta u|^2
          + \frac{C_2 D_2}{2} \intom |\nabla \Delta v|^2 \\
    &\le  (A_1 \chi_1 \ustar - A_2 \chi_2 \vstar) \intom \nabla u \cdot \nabla v
          + (B_1 \chi_1 \ustar - B_2 \chi_2 \vstar) \intom \Delta u \Delta v
          + (C_1 \chi_1 \ustar - C_2 \chi_2 \vstar) \intom \nabla \Delta u \cdot \nabla \Delta v \\
    &=    0
    \qquad \text{in $(0, T_\eta)$},
  \end{align*}
  whenever $u_0, v_0$ comply with \eqref{eq:sss:cond_init},
  where $\phi$ and $T_\eta$ are as in \eqref{eq:phi_abc:def_phi} and \eqref{eq:def_t_eta}, respectively.

  As integrating the first two equations in~\eqref{prob} implies
  $\ustar = \ol u_0 = \ol u$ and $\vstar = \ol v_0 = \ol v$ in $(0, \tmax)$,
  we further obtain by Poincar\'e's inequality~\ref{lm:poincare}
  that \eqref{eq:phi_abc:cond} is fulfilled for some $K_2 \gt 0$,
  hence the statement follows by Lemma~\ref{lm:phi_abc}.
\end{proof}

Somewhat surprisingly, also in the case \eqref{eq:sss_h2} with $\lambda_2 \mu_1 \gt \lambda_1 a_2$,
suitably choosing $A_1, A_2, B_1, B_2, C_1, C_2$ in Lemma~\ref{lm:ddt_u_2}, Lemma~\ref{lm:ddt_nabla_u_2} and Lemma~\ref{lm:ddt_delta_u_2}
allows for a cancellation of all problematic terms.
\begin{lemma}\label{lm:conv_h2_pos_pos}
  Suppose \eqref{eq:sss_h2} holds and $\lambda_2 \mu_1 \gt \lambda_1 a_2$.
  Then we can find $\eta \gt 0$ and $K_1, K_2 \gt 0$ with the property that \eqref{eq:conv_exp} holds
  whenever $u_0, v_0$ satisfy \eqref{eq:sss:cond_init}.
\end{lemma}
\begin{proof}
  Positivity of $\ustar$ and $\vstar$ implies that
  the constants
  \begin{align*}
    A_1 \defs a_2 \vstar, \quad
    A_2 \defs a_1 \ustar, \quad
    B_1 \defs (a_2 + \chi_2) \vstar, \quad
    B_2 \defs (a_1 + \chi_1) \ustar, \quad
    C_1 \defs \chi_2 \vstar
    \quad \text{and} \quad
    C_2 \defs \chi_1 \ustar
  \end{align*}
  are all positive,
  hence we may apply Lemma~\ref{lm:ddt_u_2}, Lemma~\ref{lm:ddt_nabla_u_2} and Lemma~\ref{lm:ddt_delta_u_2}
  to obtain $\eta_1 \gt 0$ such that
  \begin{align*}
    &\pe  \ddt \Big(
            \phi_{A_1, B_1, C_1}(u(\cdot, t) - \ustar)
            + \phi_{A_2, B_2, C_2}(v(\cdot, t) - \vstar)
          \Big) \\
    &\pe  + \frac{C_1 D_1}{2} \intom |\nabla \Delta u|^2
          + \frac{C_2 D_2}{2} \intom |\nabla \Delta v|^2 \\
    &\pe  + A_1 \left( -f_{u}(\ustar, \vstar) - \eta (a_1 + \mu_1) \right) \intom (u - \ustar)^2 
          + A_2 \left( -g_{v}(\ustar, \vstar) - \eta (a_2 + \mu_2) \right) \intom (v - \vstar)^2 \\
    &\le  (A_1 a_1 \ustar - A_2 a_2 \vstar) \intom (u - \ustar) (v - \vstar) \\
    &\pe  + [(A_1 \chi_1 + B_1 a_1) \ustar - (A_2 \chi_2 + B_2 a_2) \vstar] \intom \nabla u \cdot \nabla v \\
    &\pe  + [(B_1 \chi_1 + C_1 a_1) \ustar - (B_2 \chi_2 + C_2 a_2) \vstar] \intom \Delta u \Delta v \\
    &\pe  + (C_1 \chi_1 \ustar - C_2 \chi_2 \vstar) \intom \nabla \Delta u \cdot \nabla \Delta v
    \qquad \text{holds in $(0, T_\eta)$ for all $\eta \le \eta_1$},
  \end{align*}
  provided $u_0, v_0$ satisfy \eqref{eq:sss:cond_init},
  where again $\phi$ and $T_\eta$ are defined in \eqref{eq:phi_abc:def_phi} and \eqref{eq:def_t_eta}, respectively.

  Setting further $\eta_2 \defs \min\left\{\frac{-f_u(\ustar, \vstar)}{2(a_1+\mu_1)}, \frac{-g_u(\ustar, \vstar)}{2(a_2+\mu_2)}\right\}$,
  which is positive by \eqref{eq:fu_gv_gt_0},
  and noting that
  \begin{align*}
    A_1 a_1 \ustar - A_2 a_2 \vstar &= 0, \\
    (A_1 \chi_1 + B_1 a_1) \ustar - (A_2 \chi_2 + B_2 a_2) \vstar &= 0, \\
    (B_1 \chi_1 + C_1 a_1) \ustar - (B_2 \chi_2 + C_2 a_2) \vstar &= 0 \quad \text{as well as} \\
    C_1 \chi_1 \ustar - C_2 \chi_2 \vstar &= 0,
  \end{align*}
  we obtain
  \begin{align*}
    &\pe  \ddt \Big(
            \phi_{A_1, B_1, C_1}(u(\cdot, t) - \ustar)
            + \phi_{A_2, B_2, C_2}(v(\cdot, t) - \vstar)
          \Big) \\
    &\pe  + \frac{C_1 D_1}{2} \intom |\nabla \Delta u|^2
          + \frac{C_2 D_2}{2} \intom |\nabla \Delta v|^2 \\
    &\pe  - \frac{A_1 f_{u}(\ustar, \vstar)}{2} \intom (u - \ustar)^2 
          - \frac{A_2 g_{v}(\ustar, \vstar)}{2} \intom (v - \vstar)^2 \\ &\le  0
    \qquad \text{in $(0, T_\eta)$}
  \end{align*}
  for $\eta \defs \min\{\eta_1, \eta_2\}$, provided $u_0, v_0$ comply with \eqref{eq:sss:cond_init}.
  
  In virtue of Poincar\'e's inequality~\ref{lm:poincare},
  this first asserts \eqref{eq:phi_abc:cond} for some $K_2 \gt 0$
  and then also \eqref{eq:conv_exp} for some $K_1 \gt 0$ by Lemma~\ref{lm:phi_abc}.
\end{proof}

\subsection{The case \tops{\eqref{eq:sss_h2}}{(H2)} with \tops{$\lambda_2 \mu_1 \le \lambda_1 a_2$}{lambda2 mu1 <= lambda1 a2}}\label{sec:h2_le}
The condition~\eqref{eq:sss_h2} with $\lambda_2 \mu_1 \le \lambda_1 a_2$ implies $\vstar = 0$,
hence for any choice of $A_1, A_2, B_1, B_2, C_1, C_2 \gt 0$
in Lemma~\ref{lm:ddt_u_2}, Lemma~\ref{lm:ddt_nabla_u_2} and Lemma~\ref{lm:ddt_delta_u_2},
unlike as in the previous subsection, no cancellation of problematic terms occurs
(except if also $\ustar = 0$, but then we will rely on a different functional, see Lemma~\ref{lm:conv_h2_zero_zero_2} below).

However, the disappearance of $\vstar$ can also be used to our advantage.
As the coefficients of the problematic terms no longer depend on $A_2, B_2$ and $C_2$,
we can choose (one of) these parameters comparatively large and thus obtain stronger dissipative terms.
This idea first manifests itself in the following

\begin{lemma}\label{lm:ddt_delta_u_2_h2_pos_zero}
  Suppose \eqref{eq:sss_h2} holds and $\lambda_2 \mu_1 \le \lambda_1 a_2$.
  There are $\eta \gt 0$ as well as $K \gt 0$ and $C_2 \gt 0$ such that
  whenever $u_0, v_0$ comply with \eqref{eq:sss:cond_init} and $T_\eta$ is as in \eqref{eq:def_t_eta},
  \begin{align*}
          \intom |\Delta u(\cdot, t)|^2 + C_2 \intom |\Delta v(\cdot, t)|^2
    &\le  \ure^{-K t} \left( \intom |\Delta u_0|^2 + C_2 \intom |\Delta v_0|^2 \right)
  \qquad \text{for all $t \in (0, T_\eta)$}.
  \end{align*}
\end{lemma}
\begin{proof}
  Set $K \defs \frac{\min\{D_1, D_2\}}{2} \gt 0$,
  $C_1 \defs 1$ and
  \begin{align*}
    C_2 \defs \frac{16 \max\{\cp^2 a_1^2, \chi_1^2\}(\ustar+1)^2}{D_1 D_2} \gt 0,
  \end{align*}
  where $\cp \gt 0$ denotes the constant given in Lemma~\ref{lm:poincare}.

  By Lemma~\ref{lm:ddt_delta_u_2}, there is $\eta \gt 0$ with the property that
  \begin{align*}
    &\pe  \ddt \left( \intom |\Delta u|^2 + C_2 \intom |\Delta v|^2 \right)
          + D_1 \intom |\nabla \Delta u|^2 + C_2 D_2 \intom |\nabla \Delta v|^2 \\
    &\le  2 a_1 \ustar \intom \Delta u \Delta v
          + 2 \chi_1 \ustar \intom \nabla \Delta u \cdot \nabla \Delta v
    \qquad \text{in $(0, T_\eta)$},
  \end{align*}
  provided the (henceforth fixed) initial data $u_0, v_0$ satisfy \eqref{eq:sss:cond_init}.
  
  Therein are by Young's inequality and Poincar\'e's inequality~\ref{lm:poincare}, with $\cp \gt 0$ as in that lemma,
  \begin{align*}
        2 a_1 \ustar \intom \Delta u \Delta v
    &\le \frac{D_1}{4\cp} \intom |\Delta u|^2 + \frac{4 \cp a_1^2  \ustar^2}{D_1} \intom |\Delta v|^2 \\
    &\le \frac{D_1}{4} \intom |\nabla \Delta u|^2 + \frac{C_2 D_2}{4} \intom |\nabla \Delta v|^2
    \qquad \text{in $(0, \tmax)$}
  \end{align*}
  and, again by Young's inequality,
  \begin{align*}
        2 \chi_1 \ustar \intom \nabla \Delta u \cdot \nabla \Delta v
    &\le \frac{D_1}{4} \intom |\nabla \Delta u|^2 + \frac{4 \chi_1^2 \ustar^2}{D_1} \intom |\nabla \Delta v|^2 \\
    &\le \frac{D_1}{4} \intom |\nabla \Delta u|^2 + \frac{C_2 D_2}{4} \intom |\nabla \Delta v|^2
    \qquad \text{in $(0, \tmax)$}.
  \end{align*}
  Thus, the statement follows upon an integration over $(0, T_\eta)$
  due to \eqref{eq:local_ex:cont_sob22}, the $\sob22$ continuity of $u$ and $v$ up to $t=0$.
\end{proof}

In the case \eqref{eq:sss_h2} with $\lambda_2 \mu_1 \lt \lambda_1 a_2$,
by a similar argument we also obtain bounds for $\intom (u - \ustar)^2$ and $\intom v^2$.
\begin{lemma}\label{lm:conv_h2_pos_zero}
  If \eqref{eq:sss_h2} holds with $\lambda_2 \mu_1 \lt \lambda_1 a_2$,
  then there are $\eta \gt 0$, $K \gt 0$ and $A_2 \gt 0$ such that 
  \begin{align*}
          \intom (u - \ustar)^2 + A_2 \intom v^2
    &\le  \ure^{-K t} \left( \intom (u_0 - \ustar)^2 + A_2 \intom (v_0 - \vstar)^2 \right)
  \qquad \text{for all $t \in (0, T_\eta)$}.
  \end{align*}
  provided $u_0, v_0$ satisfy \eqref{eq:sss:cond_init} and $T_\eta$ is as in \eqref{eq:def_t_eta}.
\end{lemma}
\begin{proof}
  Since $\lambda_2 \mu_1 \lt \lambda_1 a_2$, both $f_u(\ustar, \vstar)$ and $g_v(\ustar, \vstar)$ are negative,
  hence there is $\eta_1 \gt 0$ such that 
  \begin{align*}
    K \defs \min\left\{-f_{u}(\ustar, \vstar) - \eta_1 (a_1 + \mu_1),  -g_{v}(\ustar, \vstar) - \eta_1 (a_2 + \mu_2)\right\} \gt 0.
  \end{align*} 
  Set moreover $A_1 \defs 1$ and
  \begin{align*}
    A_2 \defs \max\left\{ \frac{a_1^2}{K^2}, \frac{\chi_1^2}{D_1 D_2} \right\} \ustar^2 \gt 0.
  \end{align*}
  Then Lemma~\ref{lm:ddt_u_2} provides us with $\eta \in (0, \eta_1)$ such that
  \begin{align*}
    &\pe  \ddt \left( \intom (u - \ustar)^2 + A_2 \intom (v - \vstar)^2 \right) \notag \\
    &\pe  + D_1 \intom |\nabla u|^2
          + A_2 D_2 \intom |\nabla v|^2 \notag \\
    &\pe  + 2 K \intom (u - \ustar)^2 
          + 2 A_2 K \intom v^2  \notag \\
    &\le  2 a_1 \ustar \intom (u - \ustar) v
          + 2 \chi_1 \ustar \intom \nabla u \cdot \nabla v 
    \qquad \text{in $(0, T_\eta)$},
  \end{align*}
  whenever $u_0, v_0$ comply with \eqref{eq:sss:cond_init}.

  Henceforth fixing such initial data, two applications of Young's inequality give
  \begin{align*}
          2 a_1 \ustar \intom (u - \ustar) v
    &\le  K \intom (u - \ustar)^2 + \frac{a_1^2 \ustar^2}{K} \intom v^2
     \le  K \intom (u - \ustar)^2 + A_2 K \intom v^2
  \intertext{and}
          2 \chi_1 \ustar \intom \nabla u \cdot \nabla v 
    &\le  D_1 \intom |\nabla u|^2 + \frac{\chi_1^2 \ustar^2}{D_1} \intom |\nabla v|^2
     \le  D_1 \intom |\nabla u|^2 + A_2 D_2 \intom |\nabla v|^2
  \end{align*}
  in $(0, \tmax)$, so that the statement follows by the comparison principle for ordinary differential equations.
\end{proof}

The case \eqref{eq:sss_h2} with $\lambda_2 \mu_1 = \lambda_1 a_2$ cannot be handled in a similar fashion
as then $g_v(\ustar, \vstar) = 0$
resulting in the term $A_2 (-g_v(\ustar, \vstar) - \eta(a_2 + \mu_2)) \intom v^2$ in \eqref{eq:ddt_u_2:statement} having an unfavorable sign.
Similarly, if $\lambda_1 = 0$, then $f_u(\ustar, \vstar) = 0$ and $A_1 (-f_{u}(\ustar, \vstar) - \eta (a_1 + \mu_1)) \lt 0$.
Thus, we introduce an additional functional to counter these terms.
\begin{lemma}\label{lm:ddt_u}
  Suppose that $u_0, v_0$ comply with \eqref{eq:sss:cond_init}.
  If $\lambda_1 = 0$, then
  \begin{align}\label{eq:ddt_u:u}
    \ddt \intom u &= -\mu_1 \intom u^2 + a_1 \intom uv
    \qquad \text{in $(0, \tmax)$}
  \intertext{and if~\eqref{eq:sss_h2} holds with $\lambda_2 \mu_1 = \lambda_1 a_2$, then}\label{eq:ddt_u:v}
    \ddt \intom v &= -\mu_2 \intom v^2 - a_2 \intom (u - \ustar) v
    \qquad \text{in $(0, \tmax)$}.
  \end{align}
\end{lemma}
\begin{proof}
  The first statement immediately follows by integrating the first equation in \eqref{prob}.

  Furthermore, the assumptions~\eqref{eq:sss_h2} and $\lambda_2 \mu_1 = \lambda_1 a_2$
  imply $(\ustar, \vstar) = (\frac{\lambda_1}{\mu_1}, 0) = (\frac{\lambda_2}{a_2}, 0)$ and hence
  \begin{align*}
      g(u, v)
    = v (\lambda_2 - \mu_2 v - a_2 u)
    = v (\lambda_2 - \mu_2 v - a_2 \ustar) + a_2 (\ustar - u) v 
    = -\mu_2 v^2 - a_2 (u - \ustar) v
    \qquad \text{in $(0, \tmax)$}.
  \end{align*}
  Thus, the second statement follows also due to integrating.
\end{proof}

With the help of this lemma, we can now handle the remaining case,
namely \eqref{eq:sss_h2} with $\lambda_2 \mu_1 = \lambda_1 a_2$.
The proof is split into three lemmata;
before dealing with the (in some sense) fully degenerate case,
in the following two lemmata, we first handle the half-degenerate case, where at least $\ustar \gt 0$ and $f_u(\ustar, \vstar) \gt 0$.

\begin{lemma}\label{lm:conv_h2_zero_zero_1_v}
  Suppose~\eqref{eq:sss_h2}, $\lambda_2 \mu_1 = \lambda_1 a_2$ as well as $\lambda_1 \gt 0$
  and, for $\eta \gt 0$, let $T_\eta$ be as in \eqref{eq:def_t_eta}.

  \begin{enumerate}
    \item[(i)]
      There are $\eta \gt 0$ and $K_1, K_2 \gt 0$ such that
      \begin{align*}
              \|v(\cdot, t)\|_{\leb1}
        &\le  \left( K_1 \left(\|u_0 - \ustar\|_{\leb1} + \|v_0\|_{\leb1}\right)^{-1} + K_2 t \right)^{-1}
        \qquad \text{for all $t \in (0, T_\eta)$},
      \end{align*}
      whenever $u_0, v_0$ are such that \eqref{eq:sss:cond_init} holds.

    \item[(ii)]
      We can find $\eta' \gt 0$ and $K_1', K_2' \gt 0$ such that
      \begin{align*}
              \|v(\cdot, t)\|_{\sob22}
        &\le  \left( K_1' \left(\|u_0 - \ustar\|_{\sob22} + \|v_0\|_{\sob22}\right)^{-1} + K_2' t \right)^{-1}
        \qquad \text{for all $t \in (0, T_{\eta'})$},
      \end{align*}
      if $u_0, v_0$ comply with \eqref{eq:sss:cond_init}.
  \end{enumerate}
\end{lemma}
\begin{proof}
  Setting $A_1 \defs 1$, $X_2 \defs \frac{a_1 \ustar}{a_2} \gt 0$, $A_2 \defs \frac{\chi_1^2 \ustar^2}{D_1 D_2} \gt 0$,
  by Lemma~\ref{lm:ddt_u_2} and Lemma~\ref{lm:ddt_u} we find $\eta_0 \gt 0$ such that
  \begin{align}\label{eq:conv_h2_zero_zero_1:first_ddt}
    &\pe  \ddt \left(
            \frac{A_1}{2} \intom (u - \ustar)^2 + \frac{A_2}{2} \intom v^2
            + X_2 \intom v
          \right) \notag \\
    &\pe  + \frac{A_1 D_1}{2} \intom |\nabla u|^2
          + \frac{A_2 D_2}{2} \intom |\nabla v|^2 \notag \\
    &\pe  + \left( -A_1 f_{u}(\ustar, \vstar) - A_1 \eta (a_1 + \mu_1) \right) \intom (u - \ustar)^2 
          + \left( X_2 \mu_2 - A_2 \eta (a_2 + \mu_2) \right) \intom v^2 \notag \\
    &\le  \left( A_1 a_1 \ustar - X_2 a_2 \right) \intom (u - \ustar) v
          + A_1 \chi_1 \ustar \intom \nabla u \cdot \nabla v 
    \qquad \text{in $(0, T_\eta)$ for all $\eta \le \eta_0$},
  \end{align}
  whenever $u_0, v_0$ comply with \eqref{eq:sss:cond_init}.

  Set $c_1 \defs \frac{A_1 f_u(\ustar, \vstar)}{2} \gt 0$,
  $c_2 \defs \frac{X_2 \mu_2}{2} \gt 0$,
  $c_3 \defs \min\left\{\frac{4c_1}{3A_1^2}, \frac{2c_2}{3A_2^2}, \frac{c_2}{6X_2^2 |\Omega|}\right\} \gt 0$
  as well as
  \begin{align*}
    \eta \defs \min\left\{1, \eta_0, |\Omega|^{-\frac12}, \frac{c_1}{A_1 (a_1 + \mu_1)}, \frac{c_2}{A_2(a_2 + \mu_2)}\right\} \gt 0
  \end{align*}
  and fix $u_0, v_0$ satisfying \eqref{eq:sss:cond_init}.

  As the term $A_1 a_1 \ustar - X_2 a_2$ vanishes due to the definition of $A_1$ and $X_2$,
  and Young's inequality as well as the definition of $A_2$ imply
  \begin{align*}
        A_1 \chi_1 \ustar \intom \nabla u \cdot \nabla v 
    \le \frac{A_1 D_1}{2} \intom |\nabla u|^2 + \frac{A_2 D_2}{2} \intom |\nabla u|^2
    \qquad \text{in $(0, \tmax)$},
  \end{align*}
  we may conclude from \eqref{eq:conv_h2_zero_zero_1:first_ddt} that
  \begin{align*}
        \ddt \left(
           \frac{A_1}{2} \intom (u - \ustar)^2 + \frac{A_2}{2} \intom v^2
           + X_2 \intom v
         \right)
    \le  - c_1 \intom (u - \ustar)^2 - c_2 \intom v^2
    \qquad \text{holds in $(0, T_\eta)$}.
  \end{align*}
  
  Since $\eta \le |\Omega|^{-\frac12}$ implies $\intom (u - \ustar)^2 \le 1$ as well as $\intom v^2 \le 1$ in $(0, T_\eta)$
  and due to Hölder's inequality as well as the elementary inequality $(a+b+c)^2 \le 3(a^2+b^2+c^2)$ for $a, b, c \in \R$,
  we further obtain
  \begin{align*}
    &\pe  \ddt \left(
            \frac{A_1}{2} \intom (u - \ustar)^2 + \frac{A_2}{2} \intom v^2
            + X_2 \intom v
          \right) \\
    &\le  - c_1 \intom (u - \ustar)^2
          - \frac{c_2}{2} \intom v^2
          - \frac{c_2}{2} \intom v^2 \\
    &\le  - c_1 \left( \intom (u - \ustar)^2 \right)^2
          - \frac{c_2}{2} \left( \intom v^2 \right)^2
          - \frac{c_2}{2|\Omega|} \left( \intom v \right)^2 \\
    &\le  - c_3 \left(
            \frac{A_1}{2} \intom (u - \ustar)^2 + \frac{A_2}{2} \intom v^2
            + X_2 \intom v
          \right)^2
    \qquad \text{in $(0, T_{\eta})$}.
  \end{align*}
   
  Because of $\eta \le 1$
  and since without loss of generality $\|u_0 - \ustar\|_{\leb\infty} \le \eta$ and $\|v_0\|_{\leb \infty} \le \eta$,
  this implies
  \begin{align*}
          X_2 \|v(\cdot, t)\|_{\leb 1}
    &\le  \left(  
            \left(
              \frac{A_1}{2} \intom (u_0 - \ustar)^2 + \frac{A_2}{2} \intom v_0^2 + X_2 \intom v_0 \right)^{-1}
              + c_3 t
            \right)^{-1} \\
    &\le  \left(  
            \left(
              \frac{A_1}{2} \intom |u_0 - \ustar| + \left( \frac{A_2}{2} + X_2 \right) \intom v_0 \right)^{-1}
              + c_3 t
            \right)^{-1}
    \qquad \text{for all $t \in (0, T_{\eta})$}
  \end{align*}
  and hence proves part~(i) for certain $K_1, K_2 \gt 0$.

  Part~(ii) follows then from Lemma~\ref{lm:ddt_delta_u_2_h2_pos_zero}, part~(i) and the observation that
  \begin{align*}
        \|v\|_{\sob22}
    \le \|v - \ol v\|_{\sob22} + \|\ol v\|_{\leb 2}
    \le C \|\Delta v\|_{\leb 2} + |\Omega|^{-\frac12} \|v\|_{\leb1}
    \qquad \text{holds in $(0, \tmax)$}
  \end{align*}
  due to Lemma~\ref{lm:w22_delta_2} (with $C \gt 0$ as in that lemma).
\end{proof}

Next, we proceed to gain similar estimates also for the first equation.

\begin{lemma}\label{lm:conv_h2_zero_zero_1_u}
  Assume \eqref{eq:sss_h2} holds and $\lambda_2 \mu_1 = \lambda_1 a_2$ as well as $\lambda_1 \gt 0$.
  Then there are $\eta \gt 0$ and $K_1, K_2 \gt 0$ such that
  \begin{align*}
          \|u(\cdot, t) - \ustar\|_{\sob22}
    &\le  \left( K_1 \left(\|u_0 - \ustar\|_{\sob22} + \|v_0\|_{\sob22}\right)^{-1} + K_2 t \right)^{-1}
    \qquad \text{for all $t \in (0, T_{\eta})$},
  \end{align*}
  if $u_0, v_0$ satisfy \eqref{eq:sss:cond_init} and $T_\eta$ is as in \eqref{eq:def_t_eta}.
\end{lemma}
\begin{proof}
  Choose $\eta_1 \gt 0$ so small that $c_1 \defs \lambda_1 - (a_1 + \mu_1) \eta_1 \gt 0$
  and set $c_2 \defs \max\left\{\frac{a_1^2 \ustar^2}{c_1}, \frac{2\chi_1^2 \ustar^2}{3 D_1} + \chi_1\right\}$.
  By Lemma~\ref{lm:ddt_u_2_single} and Lemma~\ref{lm:conv_h2_zero_zero_1_v},
  there are moreover $\eta_2, \eta_3 \gt 0$ and $c_3, c_4 \gt 0$ such that
  \begin{align*}
    &\pe  \ddt \intom (u - \ustar)^2
          + \frac{3 D_1}{2} \intom |\nabla u|^2
          +  2 \left( -f_{u}(\ustar, \vstar) - \eta (a_1 + \mu_1) \right) \intom (u - \ustar)^2  \notag \\
    &\le  2 a_1 \ustar \intom (u - \ustar) v
          +  2 \chi_1 \ustar \intom \nabla u \cdot \nabla v 
          + \eta \chi_1 \intom |\nabla v|^2
    \qquad \text{in $(0, T_{\eta})$ for all $\eta \in (0, \eta_2]$}
  \end{align*}
  and
  \begin{align*}
        \|v(\cdot, t)\|_{\sob22}^2
    \le \left( \sqrt{c_2} c_3 \left(\|u_0 - \ustar\|_{\sob22} + \|v_0\|_{\sob22}\right)^{-1} + \sqrt{c_2} c_4 t \right)^{-2} 
    \qquad \text{in $(0, T_{\eta_3})$},
  \end{align*}
  provided $u_0, v_0$ comply with \eqref{eq:sss:cond_init}.

  Thus, fixing $\eta \defs \min\{\eta_1, \eta_2, \eta_3, 1\}$ as well as $u_0, v_0$ satisfying \eqref{eq:sss:cond_init}
  and noting that $f_u(\ustar, \vstar) = -\lambda_1$,
  we obtain
  \begin{align*}
          \ddt \intom (u - \ustar)^2
    &\le  - \frac{3D_1}{2} \intom |\nabla u|^2
          - 2 c_1 \intom (u - \ustar)^2
          + 2 a_1 \ustar \intom (u - \ustar)^2 v
          + 2 \chi_1 \ustar \intom \nabla u \cdot \nabla v
          + \eta \chi_1 \intom |\nabla u|^2 \\
    &\le  - c_1 \intom (u - \ustar)^2
          + \frac{a_1^2 \ustar^2}{c_1} \intom v^2
          + \left( \frac{2\chi_1^2 \ustar^2}{3 D_1} + \chi_1 \right) \intom |\nabla v|^2 \\
    &\le  - c_1 \intom (u - \ustar)^2
          + c_2 \left( \intom v^2 + \intom |\nabla v|^2 \right) \\
    &\le  - c_1 \intom (u - \ustar)^2
          + \left( c_3 \left(\|u_0 - \ustar\|_{\sob22} + \|v_0\|_{\sob22}\right)^{-1} + c_4 t \right)^{-2} 
    \qquad \text{in $(0, T_{\eta})$},
  \end{align*}
  which by the variation-of-constants formula implies
  \begin{align*}
          \intom (u - \ustar)^2(\cdot, t)
    &\le  \ure^{-c_1 t} \intom (u_0 - \ustar)^2(\cdot, t)
          + \int_0^t \ure^{-c_1 (t-s)} (c_3 I_0^{-1} + c_4 s)^{-2} \ds
    \qquad \text{for all $t \in (0, T_\eta)$},
  \end{align*}
  where we abbreviated $I_0 \defs \|u_0 - \ustar\|_{\sob22} + \|v_0\|_{\sob22}$.
  Noting that $[0, \infty) \ni s \mapsto (c_b I_0^{-1} + c_c s)^{-2}$ is decreasing, we further calculate
  \begin{align*}
          \int_0^t \ure^{-c_1 (t-s)} (c_3 I_0^{-1} + c_4 s)^{-2} \ds
    &=    \int_0^{t/2} \ure^{-c_1 (t-s)} (c_3 I_0^{-1} + c_4 s)^{-2} \ds
          + \int_{t/2}^t \ure^{-c_1 (t-s)} (c_3 I_0^{-1} + c_4 s)^{-2} \ds \\
    &\le  \frac{I_0^2}{c_3^2} \int_{t/2}^t \ure^{- c_1 s} \ds
          + \left(c_3 I_0^{-1} + \frac{c_4 t}{2}\right)^{-2} \int_0^{t/2} \ure^{-  s} \ds \\
    &\le  \frac{I_0^2}{c_1 c_3^2} \ure^{-\frac{c_1}{2} t}
          + \frac1{(\sqrt{c_1} c_3 I_0^{-1} + \frac{\sqrt{c_1} c_4 t}{2})^2}
    \qquad \text{for all $t \in (0, T_\eta)$}.
  \end{align*}
  Combining these estimates with Lemma~\ref{lm:ddt_delta_u_2_h2_pos_zero} and Lemma~\ref{lm:w22_delta_2}
  yields the statement for certain $K_1, K_2 \gt 0$.
\end{proof}

Finally, we deal with the aforementioned fully degenerate case.
\begin{lemma}\label{lm:conv_h2_zero_zero_2}
  Suppose~\eqref{eq:sss_h2} and $\lambda_1 = \lambda_2 = 0$.
  Then there are $\eta \gt 0$ and $K_1, K_2 \gt 0$ such that
  \begin{align}\label{eq:conv_h2_zero_zero:statement_1}
          \|u(\cdot, t)\|_{\sob22} + \|v(\cdot, t)\|_{\sob22}
    &\le  \left( K_1 \left(\|u_0\|_{\sob22} + \|v_0\|_{\sob22}\right)^{-1} + K_2 t \right)^{-1}
  \end{align}
  for all $t \in (0, T_\eta)$, where $T_\eta$ is defined in \eqref{eq:def_t_eta},
  provided $u_0, v_0$ satisfy \eqref{eq:sss:cond_init}.
\end{lemma}
\begin{proof}
  Set $c_1 \defs \frac{\min\{\mu_1, \mu_2\}}{2}$
  and fix $u_0, v_0$ complying with \eqref{eq:sss:cond_init}.

  By multiplying~\eqref{eq:ddt_u:u} and \eqref{eq:ddt_u:v} with $a_2$ and $a_1$, respectively, we obtain
  \begin{align*}
      \ddt \left( a_2 \intom u + a_1 \intom v \right)
    = - \mu_1 a_2 \intom u^2 - \mu_2 a_1 \intom v^2
    \qquad \text{in $(0, \tmax)$}.
  \end{align*}
  Hence, along with Hölder's inequality this implies
  \begin{align*}
        \ddt \left( a_2 \intom u + a_1 \intom v \right)
    \le - c_1 \left( a_2 \intom u + a_1 \intom v \right)^2
    \qquad \text{in $(0, \tmax)$},
  \end{align*}
  which upon integrating results in
  \begin{align}\label{eq:conv_h2_zero_zero:l1_conv}
        a_2 \intom u(\cdot, t) + a_1 \intom v(\cdot, t)
    \le \left( \left( a_2 \intom u_0 + a_1 \intom v_0 \right)^{-1} + c_1 t \right)^{-1}
    \qquad \text{for all $t \in (0, \tmax)$}.
  \end{align}

  As in the proof of Lemma~\ref{lm:conv_h2_zero_zero_1_v},
  we now apply Lemma~\ref{lm:w22_delta_2} (with $C \gt 0$ as in that lemma) to see that
  \begin{align*}
        \|\varphi\|_{\sob22}
    \le \|\varphi - \ol \varphi\|_{\sob22} + \|\ol \varphi\|_{\leb 2}
    \le C \|\Delta \varphi\|_{\leb 2} + |\Omega|^{-\frac12} \|\varphi\|_{\leb1}
    \qquad \text{for all $\varphi \in \con2$ with $\partial_\nu \varphi = 0$},
  \end{align*}
  which applied to $\varphi = u$ and $\varphi = v$
  and combined with \eqref{eq:conv_h2_zero_zero:l1_conv} and Lemma~\ref{lm:ddt_delta_u_2_h2_pos_zero}
  implies \eqref{eq:conv_h2_zero_zero:statement_1} for certain $K_1, K_2 \gt 0$ and $\eta \gt 0$.
\end{proof}

\section{Proof of Theorem~\ref{th:sss}}
The various lemmata from Section~\ref{sec:w22_est} allow us now to find $\eps \gt 0$
such that if $u_0, v_0$ satisfy (\eqref{eq:sss:cond_init} and) \eqref{eq:sss:cond_eps}
then $\tmax = \infty$ and $(u, v)$ converges to $(\ustar, \vstar)$.
\begin{lemma}\label{lm:main_proof}
  For $\eps \gt 0$ and $K_1, K_2 \gt 0$, define
  \begin{align*}
    y_{\eps, K_1, K_2} \colon [0, \infty) \ra \R, \quad t &\mapsto
    \begin{cases}
      (\frac1{K_1 \eps} + K_2 t)^{-1}, & \text{if \eqref{eq:sss_h2} holds and $\lambda_2 \mu_1 = \lambda_1 a_2$}, \\[1em]
      K_1 \eps \ure^{-K_2 t},  & \text{else}.
    \end{cases}
  \end{align*}
  Then there are $\eps \gt 0$ and $K_1, K_2 \gt 0$ such that $\tmax(u_0, v_0) = \infty$,
  \begin{align*}
    \|(u(u_0, v_0))(\cdot, t) - \ustar\|_{\sob22} + \|(v(u_0, v_0))(\cdot, t) - \vstar\|_{\sob22}
    &\le y_{\eps, K_1, K_2}(t)
    \qquad \text{for all $t \ge 0$},
  \end{align*}
  whenever $u_0, v_0$ satisfy \eqref{eq:sss:cond_init} and \eqref{eq:sss:cond_eps}.
\end{lemma}
\begin{proof}
  Lemma~\ref{lm:conv_h1},
  Lemma~\ref{lm:conv_h2_pos_pos},
  Lemma~\ref{lm:ddt_delta_u_2_h2_pos_zero},
  Lemma~\ref{lm:conv_h2_pos_zero}
  Lemma~\ref{lm:conv_h2_zero_zero_1_v}~(ii)
  Lemma~\ref{lm:conv_h2_zero_zero_1_u}
  and Lemma~\ref{lm:conv_h2_zero_zero_2}
  imply
  that there are $\eta \gt 0$ and $K_1, K_2 \gt 0$ with the following property:
  Let $\eps' \gt 0$.
  If $u_0, v_0$ comply with \eqref{eq:sss:cond_init} and \eqref{eq:sss:cond_eps} with $\eps$ replaced by $\eps'$, then
  \begin{align}\label{eq:main_proof:le_y}
    \|u(\cdot, t) - \ustar\|_{\sob22} + \|v(\cdot, t) - \vstar\|_{\sob22} \le y_{\eps', K_1, K_2}(t)
    \qquad \text{for all $t \in [0, T_\eta)$},
  \end{align}
  where $(u, v) \defs (u(u_0, v_0), v(u_0, v_0))$ and $T_\eta \defs T_\eta(u_0, v_0)$ is as in \eqref{eq:def_t_eta}.

  Thanks to the restriction $n \le 3$,
  Sobolev's embedding theorem asserts that there are $\alpha \in (0, 1)$ and $c_1 \gt 0$ such that
  \begin{align*}
    \|\varphi\|_{\con\alpha} \le c_1 \|\varphi\|_{\sob 22}
    \qquad \text{for all $\varphi \in \sob22$.}
  \end{align*}

  Fix an arbitrary $\eps \in (0, \frac{\eta}{c_1 \max\{K_1, 1\}})$ and $u_0, v_0$
  complying not only with \eqref{eq:sss:cond_init} but also with \eqref{eq:sss:cond_eps}.
  As then
  \begin{align*}
        \|u_0 - \ustar\|_{\leb\infty} + \|v_0 - \vstar\|_{\leb\infty}
    \le c_1 \left( \|u_0 - \ustar\|_{\sob22} + \|v_0 - \vstar\|_{\sob22} \right)
    \le c_1 \eps
    \lt \eta,
  \end{align*}
  we infer $T_\eta \gt 0$ from $u, v \in C^0(\Ombar \times [0, \tmax))$.
  Moreover,
  \begin{align}\label{eq:main_proof:u+v_lt_eta}
          \|u(\cdot, t) - \ustar \|_{\leb\infty} + \|v(\cdot, t) - \ustar \|_{\leb\infty}
    &\le  \|u(\cdot, t) - \ustar \|_{\con\alpha} + \|v(\cdot, t) - \ustar \|_{\con\alpha} \notag \\
    &\le  c_1 \left( \|u(\cdot, t) - \ustar \|_{\sob22} + \|v(\cdot, t) - \vstar\|_{\sob22} \right) \notag \\
    &\le  c_1 y_{\eps, K_1, K_2}(t) \notag \\
    &\le  c_1 y_{\eps, K_1, K_2}(0) \notag \\
    &=    K_1 c_1 \eps
    \lt   \eta
    \qquad \text{for all $t \in (0, T_\eta)$},
  \end{align}
  hence the definition~\eqref{eq:def_t_eta} of $T_\eta$ asserts $T_\eta = \tmax$.
  In that case, \eqref{eq:main_proof:u+v_lt_eta} further implies $\tmax = \infty$
  because of the blow-up criterion~\eqref{eq:local_ex:ex_crit}.
  Finally, as then $T_\eta = \tmax = \infty$, the statement is equivalent to \eqref{eq:main_proof:le_y}.
\end{proof}

Theorem~\ref{th:sss} is now a direct consequence of already proved lemmata.
\begin{proof}[Proof of Theorem~\ref{th:sss}]
  Local existence and the regularity statements were already part of Lemma~\ref{lm:local_ex},
  while global extensibility, convergence to $(\ustar, \vstar)$ as well as the claimed convergence rates
  were the subject of Lemma~\ref{lm:main_proof}.
\end{proof}

\section{Possible generalizations of Theorem~\ref{th:sss}}\label{sec:gen}
At last, let us discuss whether the methods used above,
could potentially be used to derive more general versions of Theorem~\ref{th:sss}.

\begin{remark}
  Recall that the limitation on the space dimension, namely that $n \in \{1, 2, 3\}$, has only been used at one place:
  In the proof of Lemma~\ref{lm:main_proof} we made use of the embedding $\sob22 \embed \con\alpha$ (for some $\alpha \in (0, 1)$),
  which only holds in said space dimensions.
  Thus, it is conceivable that replacing $\sob22$ by $\sob m2$ for suitable $m \in \N$ in Theorem~\ref{th:sss}
  allows for certain generalizations of our main result.

  Indeed, if $n=1$, Theorem~\ref{th:sss} remains correct
  if one replaces $\sob22$ by $\sob12$ in all occurrences (and $\sobn22$ also by $\sob12$).
  This can be seen by a straightforward modification of the proofs above:
  Combine Lemma~\ref{lm:ddt_u_2} only with Lemma~\ref{lm:ddt_nabla_u_2} and not also with Lemma~\ref{lm:ddt_delta_u_2}.
  However, a detailed proof would lead to either a considerably longer or a unreasonably more complicated exposition (or to both)
  and is hence omitted.
  
  At first glance, similar arguments as above appear to imply an analogon of Theorem~\ref{th:sss}
  (with $\sob22$ replaced by $\sob m2$ for sufficiently large $m \in \N$) even for higher dimensions.
  The main problem, however, is, that during the computations several boundary terms would appear,
  which apparently cannot be dealt with easily.
  Let us emphasize that the question whether (a suitably modified version of) Theorem~\ref{th:sss}
  holds also in the higher dimensional setting is purely of mathematical interest.
  The biologically relevant dimensions are covered in Theorem~\ref{th:sss}.
\end{remark}

\begin{remark}
  The prototypical choices of $\rho_1, \rho_2, f$ and $g$ in \eqref{prob:general}
  are mainly made for simplicity.
  We leave it to further research to determine more general conditions on these functions
  allowing for a theorem of the form of Theorem~\ref{th:sss}.
  
  Still, the methods employed should be robust enough to also allow for (certain) nonlinear taxis sensitivities, for instance.
  At least for the case \eqref{eq:sss_h2} with $\lambda_2 \mu_1 \gt \lambda_1 a_2$, however,
  the signs of $\rho_1$ and $\rho_2$ are important:
  Our approach demands, that, roughly speaking, predators move towards their prey and the prey flees from them.
  
  The case \eqref{eq:sss_h2} with $\lambda_2 \mu_1 \le \lambda_1 a_2$ is even less sensitive to such changes.
  In fact, as the proofs above clearly show, the conclusion of Theorem~\ref{th:sss}
  remains true for different signs of $\chi_1, \chi_2$
  (with the exception that for $\chi_1 \gt 0 \gt \chi_2$ or $\chi_1 \lt 0 \lt \chi_2$,
  one has to do some additional work at the level of local existence).

  Likewise, the methods presented here should, in general, also work for different functional responses.
  Again, there is one caveat:
  The species moving towards (away from) the other one needs to benefit from (be harmed by) inter-species encounters.
\end{remark}

\appendix
\section{Gagliardo--Nirenberg inequalities}
Throughout the appendix, we fix a smooth, bounded domain $\Omega \subset \R^n$, $n \in \N$ and,
for $m \in \N$ and $p \in [1, \infty)$,
set $\sobn mp \defs \complete{\{\,\varphi \in \con\infty: \partial_\nu \varphi = 0 \text{ on } \partial \Omega\,\}}{\sob mp}$.
(As can be seen easily, for $m = p = 2$, this definition is consistent with the definition of $\sobn22$ given in~\eqref{eq:sss:def_sobn22}.)

We begin by stating Poincar\'e's inequality and straightforward consequences thereof.
\begin{lemma}\label{lm:poincare}
  There exists $\cp \gt 0$ such that
  \begin{alignat*}{2}
    \intom (\varphi - \ol \varphi)^2 &\le  \cp \intom |\nabla \varphi|^2
    && \qquad \text{for all $\varphi \in \sob12$} \\
    \intom |\nabla \varphi|^2 &\le  \cp \intom |\Delta \varphi|^2
    && \qquad \text{for all $\varphi \in \sobn22$} \quad \text{and} \\
    \intom |\Delta \varphi|^2 &\le  \cp \intom |\nabla \Delta \varphi|^2
    && \qquad \text{for all $\varphi \in \sobn32$}.
  \end{alignat*}
\end{lemma}
\begin{proof}
  By Poincar\'e's inequality (cf.\ \cite[Corollary~12.28]{LeoniFirstCourseSobolev2009}), there is $\cp \gt 0$ such that
  \begin{align}\label{eq:poincare:1}
    \intom (\varphi - \ol \varphi)^2 \le \cp \intom |\nabla \varphi|^2
    \qquad \text{for all $\varphi \in \sob12$}.
  \end{align}

  By straightforward approximation/normalization arguments,
  it is sufficient to prove the remaining two inequalities
  for all $\varphi \in C^\infty(\Ombar)$ with $\intom \varphi = 0$ and $\partial_\nu \varphi = 0$ on $\partial \Omega$.
  Thus, we fix such a $\varphi$.

  An integration by parts, Hölder's inequality and \eqref{eq:poincare:1} give
  \begin{align*}
          \intom |\nabla \varphi|^2
    &=    -\intom \varphi \Delta \varphi + \int_{\partial \Omega} \varphi \partial_\nu \varphi 
     \le  \left( \intom \varphi^2 \right)^\frac12 \left( \intom |\Delta \varphi|^2 \right)^\frac12
          + 0
     \le  \left( \cp \intom |\nabla \varphi|^2 \right)^\frac12 \left( \intom |\Delta \varphi|^2 \right)^\frac12,
  \end{align*}
  hence, in both cases $\intom |\nabla \varphi|^2 = 0$ and $\intom |\nabla \varphi|^2 \gt 0$,
  \begin{align*}
    \intom |\nabla \varphi|^2 \le \cp \intom |\Delta \varphi|^2.
  \end{align*}
  
  Similarly, we have
  \begin{align*}
          \intom |\Delta \varphi|^2
    &=    -\intom \nabla \varphi \cdot \nabla \Delta \varphi + \intom \Delta \varphi \partial_\nu \varphi
     \le  \left( \cp \intom |\Delta \varphi|^2 \right)^\frac12 \left( \intom |\nabla \Delta \varphi|^2 \right)^\frac12
          + 0
     \le  \cp \intom |\nabla \Delta \varphi|^2.
    \qedhere
  \end{align*}
\end{proof}

The following lemma should also be well-known.
However, failing to find a suitable reference, we choose to give a short proof.
\begin{lemma}\label{lm:w22_delta_2}
  Let $p \in (1, \infty)$. 
  There exists $C \gt 0$ such that
  \begin{align*}
    \|\varphi - \ol \varphi\|_{\sob2p} \le C \|\Delta \varphi\|_{\leb p}
    \qquad \text{for all $\varphi \in \sobn2p$}.
  \end{align*}
\end{lemma}
\begin{proof}
  Suppose this is not the case.  
  By an approximation/normalization argument,
  there exists $(\varphi_k)_{k \in \N} \subset C^\infty(\Ombar)$
  with $\intom \varphi_k = 0$ as well as $\partial_\nu \varphi_k = 0$ on $\partial \Omega$
  and
  \begin{align*}
    \|\varphi_k\|_{\sob2p} \gt k \|\Delta \varphi_k\|_{\leb p}
    \qquad \text{for all $k \in \N$}.
  \end{align*}
  Without loss of generality, we may assume $\|\varphi_k\|_{\sob2p} = 1$ for all $k \in \N$.
  Thus, there are $\varphi_\infty \in \sob2p$ and $(k_j)_{j \in \N} \subset \N$ with $k_j \ra \infty$ for $j \ra \infty$ such that
  \begin{align*}
    \varphi_{k_j} \rh \varphi_\infty \qquad \text{in $\sob2p$ as $j \ra \infty$}.
  \end{align*}
  Since $\sob2p \embed \embed \leb p$,
  this implies
  \begin{align*}
    \varphi_{k_j} \ra \varphi_\infty \qquad \text{in $\leb p$ as $j \ra \infty$}
  \end{align*}
  and thus also $\intom \varphi_\infty = 0$.
 
  As
  \begin{align*}
          \left| \intom \nabla \varphi_\infty \cdot \nabla \psi \right|
     =    \lim_{j \ra \infty} \left| \intom \nabla \varphi_{k_j} \cdot \nabla \psi \right|
     =    \lim_{j \ra \infty} \left| \intom \Delta \varphi_{k_j} \psi \right|
     \le  \limsup_{j \ra \infty} \frac{1}{k_j} \|\psi\|_{\leb{\frac{p}{p-1}}}
    =     0
    \qquad \text{for all $\psi \in \con\infty$}
  \end{align*}
  by Hölder's inequality,
  we further conclude that $\varphi_\infty$ is constant
  and because of $\intom \varphi_\infty = 0$ we have $\varphi_\infty \equiv 0$.

  However, as \cite[Theorem~19.1]{FriedmanPartialDifferentialEquations1976} asserts
  \begin{align*}
    \|\psi\|_{\sob 2p} \le C \|\Delta \psi\|_{\leb p} + C \|\psi\|_{\leb p}
    \qquad \text{for all $\psi \in C^2(\Ombar)$ with $\partial_\nu \psi = 0$ on $\partial \Omega$}
  \end{align*}
  for some $C \gt 0$,
  we derive
  \begin{align*}
        1
    =   \lim_{j \ra \infty} \|\varphi_{k_j}\|_{\sob 2p}
    \le C \limsup_{j \ra \infty} \left( \|\Delta \varphi_{k_j}\|_{\leb p} + \|\varphi_{k_j}\|_{\leb p} \right)
    =   0,
  \end{align*}
  a contradiction.
\end{proof}

These lemmata immediately imply the following version of the Gagliardo--Nirenberg inequality.
\begin{lemma}\label{lm:gni_sob_22}
  Let $j \in \{0, 1\}$ and suppose $p, q \in [1, \infty], r \in (1, \infty)$ are such that 
  \begin{align*}
    \theta \defs \frac{\frac1p - \frac jn - \frac1q}{\frac1r - \frac2n - \frac1q} \in \left[\frac j2, 1\right).
  \end{align*}
  Then there exists $C \gt 0$ such that
  \begin{align}\label{eq:gni_sob_22:statement_1}
    \|\varphi - \ol \varphi\|_{\sob j p} \le C \|\Delta \varphi\|_{\leb r}^\theta \|\varphi - \ol \varphi\|_{\leb q}^{1-\theta}
    \qquad \text{for all $\varphi \in \sobn2r$}.
  \end{align}
  
  In particular, for any $r \in (1, \infty)$, we may find $C' \gt 0$ such that
  \begin{align}\label{eq:gni_sob_22:statement_2}
    \|\nabla \varphi\|_{\leb {2r}}^{2r} \le C' \|\Delta \varphi\|_{\leb r}^r \|\varphi - \ol \varphi\|_{\leb \infty}^r
    \qquad \text{for all $\varphi \in \sobn2r$}.
  \end{align}
\end{lemma}
\begin{proof}
  The usual Gagliardo--Nirenberg inequality~\cite{NirenbergEllipticPartialDifferential1959} gives $c_1 \gt 0$ such that
  \begin{align*}
        \|\varphi - \ol \varphi\|_{\sob j p}
    \le c_1 \|D^2 \varphi\|_{\leb r}^\theta \|\varphi - \ol \varphi\|_{\leb q}^{1-\theta} + c_1 \|\varphi - \ol \varphi\|_{\leb 1}
    \qquad \text{for all $\varphi \in \sob2r$}.
  \end{align*}
  As Hölder's inequality asserts
  \begin{align*}
    \|\psi\|_{\leb 1} \le c_2 \|\psi\|_{\leb r}^\theta \|\psi\|_{\leb q}^{1-\theta}
    \qquad \text{for all $\psi \in \leb r \cap \leb q$}
  \end{align*}
  for some $c_2 \gt 0$,
  we find $c_3 \gt 0$ such that
  \begin{align*}
    \|\varphi - \ol \varphi\|_{\sob j p} \le c_3 \|\varphi - \ol \varphi\|_{\sob2r}^\theta \|\varphi - \ol \varphi\|_{\leb q}^{1-\theta}
    \qquad \text{for all $\varphi \in \sob2r$}.
  \end{align*}
  In conjunction with Lemma~\ref{lm:w22_delta_2}, this proves \eqref{eq:gni_sob_22:statement_1}.

  Moreover, for any $r \in (1, \infty)$, letting $j \defs 1$, $p \defs 2r$ and $q \defs \infty$, we see that
  \begin{align*}
        \frac{\frac1p - \frac jn - \frac1q}{\frac1r - \frac2n - \frac1q}
    =   \frac{\frac1{2r} - \frac1n}{\frac1r - \frac2n}
    =   \frac12
    \in \left[\frac j2, 1\right).
  \end{align*}
  Hence, \eqref{eq:gni_sob_22:statement_2} follows from \eqref{eq:gni_sob_22:statement_1}.
\end{proof}

In order to avoid any discussions how $\intom |D^3 \varphi|^2$ and $\intom |\nabla \Delta \varphi|^2$
relate for $\varphi \in \sobn32$, 
we choose to prove the following Gagliardo--Nirenberg-type inequalities,
which have been used in the proof of Lemma~\ref{lm:ddt_delta_u_2}, by hand.
\begin{lemma}\label{lm:gni_sob_32}
  There exists $C \gt 0$ such that for all $\varphi \in \sobn32$
  the estimates
  \begin{align*}
    \intom |\nabla \varphi|^6 &\le C \|\varphi - \ol \varphi\|_{\leb\infty}^4 \intom |\nabla \Delta \varphi|^2
  \intertext{and}
    \intom |\Delta \varphi|^3 &\le C \|\varphi - \ol \varphi\|_{\leb\infty} \intom |\nabla \Delta \varphi|^2
  \end{align*}
  hold.
\end{lemma}
\begin{proof}
  By Lemma~\ref{lm:gni_sob_22}, there is $c_1 \gt 0$ such that
  \begin{align}\label{eq:gni_sob_32:first_sob16}
          \intom |\nabla \varphi|^6
    &\le  c_1
          \|\varphi - \ol \varphi\|_{\leb \infty}^3
          \intom |\Delta \varphi|^3
    \qquad \text{for all $\varphi \in \sobn23$}.
  \end{align}
  
  Let $\varphi \in \con3$ with $\partial_\nu \varphi = 0$ on $\partial \Omega$.
  Noting that $(|\xi| \xi)' = 2 |\xi|$ for $\xi \in \R$,
  by an integration by parts,
  Hölder's inequality
  and \eqref{eq:gni_sob_32:first_sob16} we obtain
  \begin{align*}
          \intom |\Delta \varphi|^3
    &=    \intom |\Delta \varphi| \Delta \varphi \Delta \varphi \\
    &=    - \intom \nabla(|\Delta \varphi| \Delta \varphi) \cdot \nabla \varphi \\
    &=    -2 \intom |\Delta \varphi| \nabla \varphi \cdot \nabla \Delta \varphi \\
    &\le  2
            \left( \intom |\Delta \varphi|^3 \right)^\frac13
            \left( \intom |\nabla \varphi|^6 \right)^\frac16
            \left( \intom |\nabla \Delta \varphi|^2 \right)^\frac12 \\
    &\le  2 c_1^\frac16
            \|\varphi - \ol \varphi\|_{\leb \infty}^\frac12
            \left( \intom |\Delta \varphi|^3 \right)^\frac12
            \left( \intom |\nabla \Delta \varphi|^2 \right)^\frac12,
  \end{align*}
  hence
  \begin{align*}
          \intom |\Delta \varphi|^3
    &\le  c_2
            \|\varphi - \ol \varphi\|_{\leb \infty}
            \intom |\nabla \Delta \varphi|^2,
  \end{align*}
  where $c_2 \defs 4 c_1^\frac13$.
  Plugging this into \eqref{eq:gni_sob_32:first_sob16} yields
  \begin{align*} 
          \intom |\nabla \varphi|^6
    &\le  c_1 c_2
          \|\varphi - \ol \varphi\|_{\leb \infty}^4
          \intom |\nabla \Delta \varphi|^2.
  \end{align*}
  The statement follows by an approximation procedure and by setting $C \defs \max\{c_1, c_1 c_2\}$.
\end{proof}

\small \section*{Acknowledgments}
The author is partially supported by the German Academic Scholarship Foundation
and by the Deutsche Forschungsgemeinschaft within the project \emph{Emergence of structures and advantages in
cross-diffusion systems}, project number 411007140.
Moreover, the author is grateful for various suggestions from the anonymous referees which helped to improve the present article.

\hypersetup{urlcolor=blue}

\end{document}